\documentclass[12pt]{article}

\usepackage[english]{babel}

\usepackage{amssymb}
\usepackage{amsmath}
\usepackage{amscd}
\usepackage{url}
\usepackage{tikz}
\usetikzlibrary{matrix,arrows,decorations.pathmorphing}
\usepackage{tikz-cd}

\usepackage{geometry} 
\usepackage{graphicx} 

\usepackage{booktabs} 
\usepackage{array} 
\usepackage{paralist} 
\usepackage{verbatim} 

\newtheorem{theorem}{Theorem}[section]
\newtheorem{lemma}[theorem]{Lemma}
\newtheorem{proposition}[theorem]{Proposition}
\newtheorem{corollary}[theorem]{Corollary}
\newtheorem{example}[theorem]{Example}
\newtheorem{definition}[theorem]{Definition}
\newtheorem{remark}[theorem]{Remark}
\newenvironment{proof}[1][Proof]{\begin{trivlist}\item[\hskip \labelsep {\bfseries #1}]}{\end{trivlist}}

\usepackage{enumerate}

\usepackage{soul}

\usepackage{makeidx}

\title{$\mathbb{Z}_2$ and Klein graded Lie algebras}
\author{John Tsartsaflis}
\date{September 2012}

\begin{document}
\latintext
\maketitle
\newpage
\begin{abstract}
\addcontentsline{toc}{section}{\numberline{}Abstract}
We recall the definitions and basic results for Lie superalgebras. We specify the definition for Klein graded Lie algebras and, motivated by well-known results for Lie superalgebras, we prove similar results for Klein graded Lie algebras. More precisely, we state the theorems of  Poincar\'e-Birkhoff-Witt and Ado's, as well as  Schur's lemma. Moreover, we present two examples of Klein graded algebras.
\end{abstract}
\newpage
\tableofcontents

\newpage

\section{Introduction}

This work aims to present the theory of finite dimensional Lie superalgebras and Klein graded Lie algrebas. The theory of Lie superalgebra has been developed as a generalization of Lie algebras  to include a $Z_2$ grading and it is  important in theoretical physics where they are used to describe the mathematics of supersymmetry. 

The theory of Lie superalgebras(or, as they are also called, $Z_2$ graded Lie algebras) has seen a remarkable evolution between 1974 and 2000. Lie superalgebras was firstly introduced by Corwin, Ne'eman and Sternberg in \cite{CorNeSte} as it was known in 1974. The most important result in the theory seems to be the classification by V.G. Kac of the finite dimensional simple Lie superalgebras over an algebraically closed field of characteristic zero in 1975(a sketch of this article has been given in \cite{kac2}). In 1979, M. Scheunert \cite{scheunert} gives us a mathematical description of this classification and several theorems that was omitted in \cite{kac2} are included.

In 2000 a new kind of graded Lie algebra made it's appearance. Yand Weimin and Jing Sicong \cite{yaji} introduce the Klein graded Lie algebras in order to study symmetries and supersymmetries of paraparticle systems. A significant work on graded Lie algebras has been already made by many authors including V. G. Kac \cite{kac},\cite{kac2} and M. Scheunert \cite{scheunert},\cite{scheunert2},\cite{scheunert3} as they generalize the theory of Lie algebras a step beyond Lie superalgebras. So, a lot of work was already done in the general case, however Klein graded Lie algebra theory was to be studied more comprehencively due to it's application.

Apart from presenting the theory of Lie superalgebras,we try to develop similar to Lie superalgebras' results for Klein graded case. The interest in Klein graded algebras rose by an algebra used by physicists that seems to be a Klein graded Lie algebra as we will see in section \ref{RPSexample}.

In section \ref{gradedstructures} we present the definition for a graded vector space and a G-graded Lie algebra, where G a finite abelian group. We then introduce the special cases of a Lie superalgebra and Klein graded Lie algebra. A discussion on the color of a Klein graded Lie algebra lead us to identify four types of color which are the only possible color functions for such an algebra.

In section \ref{substructure}, remark are made on the substructures of a Lie superalgebra and on the substructures of a Klein graded Lie algebra. In the following section we construct the general linear algebra which plays the same role as the general linear algebra of Lie algebras. The elements of general linear algebra are presented in matrix form.

Section \ref{derivations} is used to state and prove some propositions made by Zhang Qingcheng and Zhang Yongzheng in \cite{Erdos01} for the Klein grading case. Definitions of solvability and nilpotency are refered in these section for later use.

In section \ref{envelopingalgebra} we construct the enveloping algebra for a Lie superalgebra and for Klein graded algebra respectively and we use a variety of propositions a remarks, in the way that M. Scheunert \cite{scheunert} did, to prove the Poincar\'e-Birkhoff-Witt theorem. We keep following M. Scheunert in section \ref{representations} to prove Schur's lemma for both Lie superalgebra and Klein graded Lie algebra case. Ado's theorem is described in section \ref{inducedsrepresentations}. These resutls are already known and proven in the general case of a (G,$\theta$)-graded Lie algebra, where $\theta$ an arbitrary color and G a finite abelian group, thanks to  V. G. Kac and M. Scheunert.

Finally, in the last two sections, examples of Klein graded Lie algebras are studied. We examine the special case of a Klein graded Lie algebra $L=L_e\oplus L_r \oplus L_s \oplus L_t$ with $sl(2,C)$ being the $L_e$ part of it and we try to find some resutls based on that fact. An example taken from mathematical physics constitutes our last section, the relative parabose set, which was the inspiring example for studying Klein graded  Lie algrebas in the first place.

In this work, original result like simple calculations or propositions for Klein graded Lie algebras are introduced in the following. In section 5, proposition \ref{53prop}. In section 7 proposition \ref{76prop}, lemma \ref{71lemma} and remark \ref{712remark}. In section 9, all calculations and basics results are original.  Let $L=L_e \oplus L_r \oplus L_s \oplus L_t$ is a Klein graded algebra, we set $L_e = sl(2,\mathbb{C})$. This way may lead someone to classify all Klein graded Lie algebras which have a simple Lie algebra as the $L_e$ part. This is only a workaround, but it looks promising.

\section{General on graded structures}\label{gradedstructures}

\begin{remark}
All vector spaces and algebras we are dealing with, they are over a field K of characteristic zero and they are finite dimensional.
\end{remark}
\begin{definition}
Let $K$ be a field, and let $V$ be a vector space over $K$ equipped with an additional binary operation from $V \times V \to V$, denoted here by  $\cdot$ (i.e. if x and y are any two elements of $V, x \cdot y$ is the product of $x$ and $y$). Then $V$ is an algebra over $K$ if the following identities hold for any three elements $x, y$, and $z$ of $V$, and all elements ("scalars") $a$ and $b$ of $K$:
\[ (x + y) \cdot z = x \cdot z + y \cdot z \]
\[ x \cdot (y + z) = x \cdot y + x \cdot z \]
\[ (ax) \cdot (by) = (ab)(x \cdot y)\]

We call the algebra $V$ assiciative if
\[ (x\cdot y)\cdot z = x \cdot (y \cdot z)\]

We call the algebra $V$ unital if there is an element $e \in V$:
\[ e \cdot x = x \cdot e = x \]
for all $x \in V$
\end{definition}

The associative algebras appearing in this work will always contain a unit element.

\begin{definition}
Let V be a vector space over the field $\mathcal{K}$ and G an abelian group. In this paper the gradation group is always assumed an abelian group. A G-gradation of the vector space V is a family $(V_g)_{g \in G }$ of sub spaces of V such that 

\begin{equation}
V = \bigoplus_{g \in G} V_g
\end{equation}
\end{definition}

An element of V is called homogeneous of degree g, $g \in G$, if it is an element of $V_g$. In the case G = $\mathbb{Z}_2$ the elements of $V_0$ are also called even, and the elements of $V_1$ are called odd.

Every element of $x \in V$ has a unique decomposition of the form
\begin{equation}
x = \sum_{g \in G} x_g   \hspace{20pt} x_g \in V_g, g \in G
\end{equation}
In the above relation, only finitely many $x_g$ are different from zero. The element $x_g$ is called the homogeneous component of x of degree g.

A subspace of U of V is called G-graded if it contains the homogeneous components of all its elements, i.e. if 
\begin{equation}
U = \bigoplus_{g \in G} (U \cap V_g)
\end{equation}

\begin{definition}
Let W = $\bigoplus_{g \in G} W_g$ be a second $G$-graded vector space. A linear mapping
\begin{equation}
f: V \to W
\end{equation}
is said to be homogeneous of degree $g \in G$, if 
\begin{equation}
f(V_q) \subset W_{qg} \hspace{20pt} \forall q \in G.
\end{equation}

The mapping f is called a homomorphism of the $G$-graded vector space $V$ into the $G$-graded vector space $W$ if and only if is homogeneous of degree $e$(the identity element of $G$). The mapping f is called an isomorphism if it is a bijective homomorphism.
\end{definition}

\begin{definition}
Let A be an algebra over $\mathcal{K}$. The algebra A is said to be G-graded if  the underlying vector space is G-graded,
\begin{equation}
A = \bigoplus_{g \in G} A_g
\end{equation}
and if, furthermore,
\begin{equation}
A_g A_q \subset A_{gq} \hspace{20pt} \forall q,g \in G
\end{equation}
Evidently, $A_0$ is a subalgrbra of $A$. If A has a unit element then this element Lies in $A_0$.
\end{definition}

\begin{definition}
A left ideal of a algebra is a linear subspace that has the property that any element of the subspace multipLied on the left by any element of the algebra produces an element of the subspace. In symbols, we say that a subset $L$ of a algebra $A$ is a left ideal if for every $x$ and $y$ in $L, z$ in $A$ and c in $K$, we have the following three statements.
\begin{align*}
&x + y \in L \text{(L is closed under addition) }\\
& cx  \in L \text{(L is closed under scalar multiplication) }\\
&zx  \in L \text{(L is closed under left multiplication by arbitrary elements)}
\end{align*}
A two-sided ideal is a subset that is both a left and a right ideal. The term ideal will be used to mean a two-sided ideal.
\end{definition}

A homomorphism of $G$-graded algebras is by definition a homomorphism of the underlying algebras as well as of the underlying $G$-graded vector spaces:
\begin{definition}
A homomorphism of degree g between two graded algebras A and B,over a field $K$,is a map $F: A \to B$ such
\begin{align*}
F(kx) &= kF(x) \\
F(x + y) &= F(x) + F(y) \\
F(xy) &= F(x)F(y) \\
F(x_r) & \subset B_{rg}
\end{align*}
for all $k \in K$ , $x,y \in A$ , $x_r \in A_r$ , $ r,g \in G$.
\end{definition}

A graded subalgebra of a $G$-graded algebra $A$ is a subalgebra of the algebra A which is, in addition, a graded subspace of the G-graded vector space $A$. A graded ideal of a $G$-graded algebra $A$ is an ideal of the algebra $A$ which is, in addition, a graded subspace of the $G$-graded vector space $A$. 
The quotient algebra of a G-graded algebra modulo a (two-side) graded ideal is again a G-graded algebra.

\begin{definition}
Let G be a group then a map $\theta : \mathbf{G} \times \mathbf{G} \to \mathbb{C^*} $ is called a color map if
\begin{itemize}\addtolength{\itemsep}{-0.6\baselineskip}
	\item $\theta (a,b) \theta(b,a) = 1$
	\item $\theta (a,bc) = \theta(a,b) \theta(a,c) \hspace{20 pt}, \forall a,b,c \in \mathbf{G}$
	\item $\theta (ab,c) = \theta(a,c) \theta(b,c) \hspace{20 pt}, \forall a,b,c \in \mathbf{G}$
\end{itemize}
\end{definition}

\noindent
\begin{definition}
An algebra $\mathbf{L}$  is a $(G,\theta)$-Graded Lie algebra iff:
\begin{itemize}\addtolength{\itemsep}{-0.6\baselineskip}
	\item $G$ is an abelian group
	\item $\exists L_a, a \in G$ such as $L = \bigoplus_{a \in G} L_a$
	\item $\exists \lbrack  \hspace{2 pt}, \rbrack: L \times L \to L \hspace{5 pt}$ such as
	\begin{enumerate}[i.]
	\item $[A,B] = -\theta(a,b)[B,A] \hspace{20 pt}$ for all homogeneous elements A,B $\in L$
	\item $0 = \theta(z,x)[X,[Y,Z]] + \theta(y,z)[Z,[X,Y]] + \theta(x,y)[Y,[Z,X]] \hspace{10 pt}$  (Jacobi identity)
 \linebreak for all homogeneous elements  X,Y,Z $\in L$ 
	\item $[L_a,L_b] \subseteq L_{ab}$ , for all $a,b \in G$
	\end{enumerate}

\end{itemize}
\end{definition}
\noindent
If an element A $\in L_{a}$ we call it homogeneous of degree a.

\begin{example}
A Lie algebra is a trivial graded algebra. The gradation group is the trivial group \{e\} containing only the identity element and the gradation map $\theta =1$.
\end{example}
\begin{example}
A Lie superalgebra is a $(Z_2,\theta)$-graded Lie algebra, the map $\theta$ will be described below for this case.
\end{example}

 From now on, we will use capital letter for the elements of an algebra, and small letter for their degree.

\begin{definition}
 A homomorphism between two graded Lie algebras $L,L'$ (over the same ground field) is a linear map that is compatible with the commutators:
\begin{equation}
f: L \to L', \hspace{20pt} f[x,y] = [f(x),f(y)]
\end{equation}
for all x,y in L.
\end{definition}

\begin{remark}
From now on, we shall call a homomorphism(without refering the degree) a homomorphism of degree e.
\end{remark}

\subsection{Superalgebra case}\label{superalgebra}
First of all, a superalgebra is a $Z_2$-graded algebra, the physicists use the prefix super, it comes from the theory of supersymmetry in theoretical physics.  Two colors can be  defined on $\mathbb{Z}_2\times\mathbb{Z}_2$. Using properties of a color map we get that
\[
\theta(a,a)^2 = 1 \hspace{20pt},\forall a \in \mathbb{Z}_2
\]
and
\[
\theta(a,0) = 1 \; \forall a \in \mathbb{Z}_2
\]
Hence, $\theta(a,0)$  = 1. We don't have any rescrictions for $\theta(1,1)$. Consider $\epsilon \in $ \{-1,1\}, then every color on $\mathbb{Z}_2\times\mathbb{Z}_2$ can be described as
$$\begin{array}{|c|c|c|}
\hline
\theta & 0 & 1\\
\hline
0 & 1 & 1 \\
\hline
1 & 1 & \epsilon \\
\hline
\end{array}$$

If L is a ( $\mathbb{Z}_2\times\mathbb{Z}_2$, $\theta$)-graded Lie algebra, in the case that $\epsilon$ = 1 we have an ordinary Lie algebra.

If L is a ( $\mathbb{Z}_2\times\mathbb{Z}_2$, $\theta$)-graded Lie algebra, in case that $\epsilon$ = -1 we have a Lie superalgebra.

\begin{definition}
Let L = $L_0 \bigoplus L_1$ be a superalgebra whose multiplication is denoted by a pointed bracket [ , ]. This impLies in particular that for all a,b $ \in \mathbb{Z}_2$
\begin{equation}
[L_a,L_b] \subset L_{a+b}
\end{equation}
We call L a Lie superalgebra if the multiplication satisfies the following identities:
\begin{equation}
[A,B] = -(-1)^{ab}[B,A]
\end{equation}
\begin{center}
(graded skew-symmetry)
\end{center}
\begin{equation}
(-1)^{ac}[A,[B,C]] + (-1)^{ab}[B,[C,A]] + (-1)^{bc}[C,[A,B]] =0
\end{equation}
\begin{center}
(graded jacobi identity)
\end{center}
for all A $\in L_a$, B $\in L_b$, C $\in L_c$ ; a,b,c $\in  \mathbb{Z}_2$.
\end{definition}

\subsection{Klein-Graded case}\label{kleingraded}

Lets recall the Klein group's basic facts. It's a group isomorphic to $\mathbb{Z}_2 \bigoplus \mathbb{Z}_2$  and if we use the symbols \{e,s,t,r\} for it's elements, then the Cayley table for this group is given by:

$$\begin{array}{|c||c|c|c|c|}
\hline
. & e & r & s & t \\
\hline\hline
e & e & r & s & t \\
\hline
r & r & e & t & s \\
\hline
s & s & t & e & r \\
\hline
t & t & s & r & e \\
\hline\end{array}$$
\noindent

If a color $\theta : \mathbf{K \times K} \to \mathbb{C^*}$, and $\mathbf{K} $ is the Klein group, then we can see that this color function takes some specific forms. By the second and third property of the color map definition, if we set a=b=c we can get that $\theta(e,a)=\theta(a,e)=1 \hspace{10 pt},\forall a\in \mathbf{K}$.
Also, using the second property we can see that:

\[ \theta(a,b)^2 = \theta(a,b)\theta(a,b) = \theta(a,b^2)=\theta(a,e)=1 \newline \Rightarrow \theta(a,b)=\pm1\]
\[ \Rightarrow \theta(a,b) = \theta(b,a)\]
Furthermore, \hspace{5 pt} \[\theta(ab,ab)=\theta(ab,a)\theta(b,b)\theta(a,a)\theta(a,b) = \theta(a,a)\theta(b,b)\]
\newline So, if we set a=r, b=s, we conclude that:
\[\theta(t,t)=\theta(r,r)\theta(s,s)\]
\[\Rightarrow \theta(r,r)\theta(s,s)\theta(t,t)=1\]
and the possible colors are the following:

$$\begin{array}{c c}
$$\begin{array}{|c||c|c|c|c|}
\hline
\theta & e & r & s & t \\
\hline\hline
e & 1 & 1 & 1 & 1 \\
\hline
r & 1 & 1 & \epsilon & \epsilon \\
\hline
s & 1 & \epsilon & 1 & \epsilon \\
\hline
t & 1 & \epsilon & \epsilon & 1 \\
\hline\end{array}$$
 & 
$$\begin{array}{|c||c|c|c|c|}
\hline
\theta & e & r & s & t \\
\hline\hline
e & 1 & 1 & 1 & 1 \\
\hline
r & 1 & -1 & \epsilon & -\epsilon \\
\hline
s & 1 & \epsilon & -1 & -\epsilon \\
\hline
t & 1 & -\epsilon &-\epsilon & 1 \\
\hline\end{array}$$
  \\
\end{array}$$
\noindent
with $\epsilon \in \{-1,1\}$.
\newline
Furthermore, we can see four cases about the color:
\newline
$$\begin{array}{c c c c}
$$\begin{array}{|c||c|c|c|c|}
\hline
\theta_1 & e & r & s & t \\
\hline\hline
e & 1 & 1 & 1 & 1 \\
\hline
r & 1 & 1 & 1 & 1 \\
\hline
s & 1 & 1 & 1 & 1 \\
\hline
t & 1 & 1 & 1 & 1 \\
\hline\end{array}$$
 & 
$$\begin{array}{|c||c|c|c|c|}
\hline
\theta_2 & e & r & s & t \\
\hline\hline
e & 1 & 1 & 1 & 1 \\
\hline
r & 1 & 1 & -1 & -1 \\
\hline
s & 1 & -1 & 1 & -1 \\
\hline
t & 1 & -1 &-1 & 1 \\
\hline\end{array}$$
& 
$$\begin{array}{|c||c|c|c|c|}
\hline
\theta_3 & e & r & s & t \\
\hline\hline
e & 1 & 1 & 1 & 1 \\
\hline
r & 1 & -1 & 1 & -1 \\
\hline
s & 1 & 1 & -1 & -1 \\
\hline
t & 1 & -1 &-1 & 1 \\
\hline\end{array}$$
& 
$$\begin{array}{|c||c|c|c|c|}
\hline
\theta_4 & e & r & s & t \\
\hline\hline
e & 1 & 1 & 1 & 1 \\
\hline
r & 1 & -1 & -1 & 1 \\
\hline
s & 1 & -1 & -1 & 1 \\
\hline
t & 1 & 1 &1 & 1 \\
\hline\end{array}$$
  \\
\end{array}$$
\newline

We can find a map from Klein group to $Z_2$, $f:\mathcal{K} \times \mathcal{K} \to Z_2$ ($Z_2$ as a ring)  in order to describe each of these color in the following way: $\theta = (-1)^f$. If we consider the elements of  $\mathcal{K}$ as a pair $(a,b)$ such as e=(0,0), r=(1,0), s=(0,1), t=(1,1), then it's easy to see that the map for the first color is \[f_1((a_1,b_1),(a_2,b_2))=0\] for the second it is  \[f_2((a_1,b_1),(a_2,b_2))=a_1b_2 + a_2b_1\] for the third it is  \[f_3((a_1,b_1),(a_2,b_2))=a_1a_2 + b_1b_2\] for the last one it is  \[f_4((a_1,b_1),(a_2,b_2))=(a_1 + b_1)(a_2+b_2)\].

Note that the previous calculations are modulo 2.

\begin{definition}
Let $\theta$ be a color map on the Klein group and let $\mathcal{K}$ be the Klein group, then a  Klein graded Lie algebra is a ($\mathcal{K}$,$\theta$)-graded Lie algebra. This impLies in particular that for all a,b $ \in \mathcal{K}$
\begin{equation}
[L_a,L_b] \subset L_{ab}
\end{equation}
and
\begin{equation}
[A,B] = -\theta(a,b)[B,A]
\end{equation}
\begin{center}
(graded skew-symmetry)
\end{center}
\begin{equation}
\theta(a,c)[A,[B,C]] + \theta(a,b)[B,[C,A]] +\theta(b,c)[C,[A,B]] =0
\end{equation}
\begin{center}
(graded jacobi identity)
\end{center}
for all A $\in L_a$, B $\in L_b$, C $\in L_c$ ; a,b,c $\in  \mathcal{K}$.
\end{definition}

\begin{remark}
We can rewrite Jacobi identity in the following way
\[ [A,[B,C]] = [[A,B],C] + \theta(a,b)[B,[A,C]] \]
This equation is known as the Leibniz rule.
\end{remark}

Let $L = L_e \oplus L_r \oplus L_s \oplus L_t$ be a Klein graded Lie algebra, we can obtain the following remarks.

\begin{remark} $L_e$ is a Lie algebra. The elements of $L_e$ have degree e, and $\theta(e,e)$ = 1. If we recall the relations that define a Klein graded Lie algebra and consider the restriction on $L_e$, we can see that they are the same as the definition of a Lie algebra in the case of $L_e$.
\end{remark}

\begin{remark} 
If the graded vector space L is equipped with "inverted multiplication"
\begin{equation*}
(A,B) \to [B,A]
\end{equation*}
then we have:
\[(A,B) = [B,A] = -\theta(b,a)[A,B] = -\theta(b,a)(B,A) = -\theta(a,b)(B,A)\]
Jacobi identity also holds due to the same fact that for the Klein group:
\[\theta(a,b) = \theta(b,a)\]
So we obtain again a Klein graded Lie algebra. 
\end{remark}
The same remarks imply on superalgebra case.

\section{Substructures}\label{substructure}

\subsection{Subalgebras and Ideals}

 Let $L=L_e \oplus L_r \oplus L_s \oplus \L_t$ be a Klein graded Lie Algebra. 
\begin{definition}
A Subalgebra $W=W_e \oplus W_r \oplus W_s \oplus W_t$ is a subset of elements of L which forms a vector subspace of L that is closed with respect to the Lie product of L such that $W_e \subset L_e$, $W_r \subset L_r$, $W_s \subset L_s$ and $W_t \subset L_t$. A subalgebra W of L such that $W \neq L$ is called a proper subalgebra of L.
\end{definition}

\begin{definition}
An ideal I of L is a subalgebra of L  such that $[L,I] \subset I$, that is 
\linebreak
$x \in L, y \in I \rightarrow [x,y] \in I$. An ideal I of L such that $I \neq L$ is called a proper ideal of L.
\end{definition}

\begin{proposition}
Let $I$ be an ideal of $L$. $I_e$ is an ideal of the Lie algebra $L_e$.
\end{proposition}
\begin{proof}
If $I$ is an ideal of $L$ then
\[ I = I_e \oplus I_r \oplus I_s \oplus I_t \]
and
\[[I_e,L_e] \subset I_e.\]
\end{proof}

\begin{proposition}
Let $I$ be an ideal of $L$. $I_e\oplus I_r$  is an ideal of the  graded algebra $L_e\oplus L_r$.
\end{proposition}
\begin{proof}
If $I$ is an ideal of $L$ then
\[ I = I_e \oplus I_r \oplus I_s \oplus I_t \].
Hence,
\[[I_e,L_e] \subset I_e\]
\[[I_e,L_r] \subset I_r\]
\[[I_r,L_e] \subset I_r\]
\[[I_r,L_r] \subset I_e\]
and
\[ [I_e\oplus I_r, L_e\oplus L_r] \subset I_e\oplus I_r\]
\end{proof}

\begin{proposition}
Let  $I,J$ two ideals of $L$. Then $I+J$, $I \cap J$, $[I,J]$ are ideals of $L$.
The center of $L$: $\mathcal{Z}(L) = \{z \in L : [z,L] = 0\}$ is an ideal.
\end{proposition}
\begin{proof}
Let z $\in$ L, x $\in$ I, y $\in$ J.

%
%
%

For the sum, due to linearity we have that
\begin{equation}
[z,x+y] = [z,x] + [z,y]
\end{equation}
I,J are ideals, so $[z,x]$ is in $I$, and $[z,y]$ is in $J$ so $[z,x] + [z,y]$ is in $I + J$ and $I + J$ is an ideal.

In the case of intersection, let $x \in I\cap$J, $z \in L$. If we consider $x$ as an element of $I$ we have
\begin{equation}
[z,x] \in I
\end{equation}
if we consider $x$ as an element of $J$ we have
\begin{equation}
[z,x] \in J
\end{equation}
finally we get that $[z,x]$ is in the intersection ideal.

In the case of the bracket ideal, we have that
\begin{equation}
[z,[x,y]] = [[z,x],y] + \theta(z,x)[x,[z,y]]
\end{equation}
$I,J$ are ideals so we have that $[z,x]$ and $[z,y]$ are in $I$ and $J$ respectivly. Hence, $[[z,x],y] \in [I,J]$ and $[x,[z,y]] \in [I,J]$, consequently $[I,J]$ is an ideal.

Finally, the center of L is an ideal. Let x in $\mathcal{Z}(L)$ and z in L. We have \[[z,x] = 0\] so\[[y,[z,x]] =0\] for every y in L. Hence, [z,x] is in $\mathcal{Z}(L)$, so $\mathcal{Z}(L)$ is an ideal.
\end{proof}

\begin{theorem}
Let a be an graded ideal of L, and let $\pi : L \to L/a$ be the natural projection.Let $L'$ be another Klein graded Lie algebra and suppose that $f : L \to L'$  is a homomorphism such that $a \subset kerf$. Then there is a unique homomorphism $ f' : L/a \to L'$ satisfying $ f'  \circ \pi = f$.
\end{theorem}
\begin{proof}
Notice that
\[\pi(L_x) = L_x + a \]
for all x in Klein group. Therefore, $\pi$ is a homomorphism of degree e. 
Let $f': L/a \to L'$ be given by $f'(x + a) = f(x)$, for all $x \in L$. 
Since f(a) = 0, $f'$ is well-defined. It is a homomorphism since 
\[
f'[x + a,y + a] = f'([x,y]+a) = f[x,y] = [f(x),f(y)] = [f'(x+a), f'(y+a)]
\]
and the degree of $f'$ is e
\[f'(L_x + a) = f(L_x) \subset L_x\]
for all x in Klein group.
Finally, it is unique since the condition $f ' \circ \pi = f$ means that, for all $x \in L$, 
\[f ' (x+a) = f ' \circ \pi (x) = f(x).\]
\end{proof}

\begin{corollary}
 Let $f : L \to L'$  be a homomorphism of  Klein graded Lie algebras. Then the image f(L) is a Klein graded Lie subalgebra of $L'$, and the resulting map $f' : L/kerf \to f(L)$ is an isomorphism. Thus, if $f$ is onto, then $f'$ is an isomorphism of $L/kerf$ onto $L'$.
\end{corollary}

\begin{proof}
 For any $x, y \in L$, we have 
\[[f(x),f(y)] = f[x,y] \in f(L)\]
so $f(L)$ is a subalgbra of $L'$. Put $a = kerf $ in the previous theorem. Then $f'$ is injective since if $f'(x + a) = 0$, then $f(x) = 0$, so $x \in a$, and thus x + a = a. Thus the map $f' : L/kerf \to f(L)$ is an isomorphism.
\end{proof}

\begin{theorem}
 Let $L, L'$ be Klein graded Lie algebras and let f : $L \to L'$ be a surjective homomorphism.
\begin{enumerate}
\item If a is a graded ideal of $L$, then f(a) is a graded ideal of $L'$.
\item If s is an graded ideal of $L'$, then $f^{-1}(s)$ is an graded ideal of $L$ which contains $kerf$.
\item The mappings a $\to$ f(a) and s $\to  f^{-1}(s)$ are inverse mappings between the set of all graded ideals of $L$ which contain $kerf$ and the set of all graded ideals of $L'$, so the two sets of graded ideals are in one-to-one correspondence.
\item L/a $\cong L'$/f(a) for all ideals a of L containing kerf.
\end{enumerate}
\end{theorem}
\begin{proof}
$\hspace{10pt}$ \newline
\begin{enumerate}
\item 
For any y $\in L'$ and v $\in$ a, we have y = f(x) for some x $\in L$, so
\begin{equation}
[y,f(v)] = [f(x),f(v)]
= f[x,v] \in f(a)
\end{equation}
Hence [$L'$,f(a)] $\subset$ L(a), and $f(a)$ is an ideal of $L'$. The gradation of $f(a)$ comes in a natural way, we set
\[f(a_e) = f(a)_e\]
\[f(a_r) = f(a)_r\]
\[f(a_s) = f(a)_s\]
\[f(a_t) = f(a)_t\]
\item
Let v $\in f^{-1}(s)$. Then for any x $\in L$, we have
\begin{equation}
f[x,v] = [f(x),f(v)] \in [L',s] \subset s,
\end{equation}
so [x,v] $\in f^{-1}(s)$. $Kerf$ of course is in $f^{-1}(s)$ because 0 is in s.
\item
We first claim that if a is an ideal of L containing kerf, then $f^{-1}(f(a))$ =a. Since clearly a $\subset f^{-1}(f(a))$, it suffices to prove that $f^{-1}(f(a)) \subset a$. But if x $\in f^{-1}(f(a))$, then f(x) = f(v) for some v $\in$ a, so x - v $\in kerf$, and hence x $\in$ v + $kerf \subset$ a + a = a. Next, it is clear from the surjectivity of $f$ that if A is any subset of $L'$, then $f(f^{-1}(A))$ = A. Thus, in particular, if s is an ideal of $L'$, then $f(f^{-1}(s))$ = s.
\item
Consider the following diagram of homomorphisms
\[
\begin{CD}
L @>f>> L'\\
\pi @VVV  @VV\pi ' V\\
L /a @>g>> L' / f(a)
\end{CD}
\]
where $\pi$ and $\pi '$ are projections. Now $\pi \circ f$  is a homomorphism of L onto
$L ' /f(a)$ whose kernel is $f^{-1}(f(a))$ = a. Hence, by the previous theorem and its corollary, there is a isomorphism g from
L/a onto $L'/f(a)$ such that $g \circ \pi = \pi ' \circ f$.
\end{enumerate}
\end{proof}
The propositions and theorems in this section are also true for a Lie superalgebra. As you can see  the bracket or the grading of the algebra are not involved in the proof.
\subsection{Lie superalgebra}

Let $L=L_0 \bigoplus L_1$ be a Lie superalgebra. Then it's clear than $L_0$ is a Lie algebra. In $L_0$ we have that
\[
[x,y] = -[y,x]
\]
\[
[x, [y, z]] + [y, [z, x]] + [z, [x, y]] = 0
\]
for all x,y,z in $L_0$.

Since $L_0$ is a Lie algebra, and the fact that the grading of L give us the rule $[L_0,L_1] \subset L_1$, we notice that $L_1$ can be considered as a module of $L_0$.
\subsection{Klein graded Lie algebra}
Let $L= L_e \oplus L_r \oplus L_s \oplus L_t$ be a Klein graded Lie algebra. This case is more complicated, and we have more interesting substructures here.
\subsubsection{Lie algebras}
The subspace $L_e$ is a Lie Algebra no matter what color we use. This is true, $L_e$ is a vector space, the restriction of L-bracket on $L_e$  satisfies the bilinearity, alternating and the jacobi identity needed for a Lie Algebra.
The same fact holds for the entire algebra L, if we us the color $\theta_1$.This color is trivial so it's easy to see that L with the color $\theta_1$ is a Lie algebra.  Lie algebras can also be found in case that the color we use isn't $\theta _1$, i.e. the  subalgebras $L_e \bigoplus L_r$, $L_e \bigoplus L_s$, $L_e \bigoplus L_t$ are Lie algebras if we use $\theta _2$ (L isn't a Lie algebra this time). In case of $\theta _3$ and $\theta _4$, $L_e \bigoplus L_t$ is a Lie algebra.
\subsubsection{Lie superalgebras}\label{Liesuperalgebras}
We can find Lie superalgebra structure on certain subalgebras of a $(\mathcal{K},\theta)$-Graded Lie Algebra. In more details, we have the following superalgebras depending on the color we choose:
\newline Case of $\theta_3$ color: The subalgebras $L_e \bigoplus L_r$ ,  $L_e \bigoplus L_s$, are superalgebras.
\newline Case of $\theta_2$ color: The subalgebras $L_e \bigoplus L_r$ ,  $L_e \bigoplus L_s$, are superalgebras.
\newline On both cases, they are superalgebras because the respective colors are the same as a $\mathbb{Z}_2$-Graded Lie algebra requires.

\subsubsection{Module substructure}
Since $L_e$ is a Lie algebra, and the fact that the grading of L give us the rules 
\[[L_e,L_r] \subset L_r , [L_e,L_s] \subset L_s , [L_e,L_t] \subset L_t \]
 we notice that $L_r, L_s, L_t$ can be considered as a module of $L_e$. More genericl, every subset of L that does not contain elements of $L_e$ is a $L_e$-module.

Except $L_e$ modules, we can see that
\[[L_r,L_s] \subset L_t , [L_r,L_t] \subset L_s \]
hence
\[[L_e \oplus L_r, L_s \oplus L_t] \subset L_s \oplus L_t\]
so $ L_s \oplus L_t$ is an $L_e \oplus L_r$ -module. $L_r$ was randomly selected, we also have $ L_e \oplus L_s$ and $ L_e \oplus L_t$ modules in the same way.

\section{General Linear Algebra}\label{gla}

If $V$ be a Klein graded vector space then $End(V)$ is a K-Graded Lie Algebra. That's it:
\begin{equation}
End(V) =  \{ h:V \to V / h: \text{linear map}\}
\end{equation}
We can define a grading on $End(V)$ as the following one:
\begin{equation}\label{endvdefinition}
End(V)_a =  \{ h:V \to V / h(V_b) \subset V_{ab}, b \in \mathcal{K} \}
\end{equation}
for all a in $\mathcal{K}$.

There is a natural way of defining a bracket [ , ] in the algebra $End(V)$  by the equality:
\begin{equation}\label{endvsymmetric}
[A,B] = AB - \theta(a,b)BA
\end{equation}
For an associative algebra($End(V)$ is one of that kind) we have the following important identity:
\begin{equation}\label{endvjacobi}
[A,BC] = [A,B]C + \theta(a,b)B[A,C]
\end{equation}

\begin{proof}

\begin{eqnarray}
 [A,B]C + \theta(a,b)B[A,C] &=& (AB - \theta(a,b)BA)C + \theta(a,b)B(AC + \theta(a,c)CA) \nonumber \\
&=& ABC + \theta(a,b)\theta(a,c)BCA=ABC + \theta(a,bc)BCA  \nonumber \\
&=& [A,BC]
\end{eqnarray}
\end{proof}
In case that we have a $Z_2$-graded vector space, then $End(V)$ is a Lie superalgebra and the relations \eqref{endvdefinition} , \eqref{endvsymmetric}, \eqref{endvjacobi} switch to
\begin{equation}
End(V)_a =  \{ h:V \to V / h(V_b) \subset V_{a+b}, b \in \mathbb{Z}_2 \}
\end{equation}
\begin{equation}
[A,B] = AB - (-1)^{ab}BA
\end{equation}
\begin{equation}
[A,BC] = [A,B]C + (-1)^{ab}B[A,C]
\end{equation}
respectivly.

It's easy but usefull to see the elements of this algebra. The $A_{ij}$ in each case, are matrices of same dimensions as the dimensions of the respective vector spaces  i.e. if we number e,r,s,t $\to$ 1,2,3,4 we can say that  $A_{ij}$ is a $(dim(V_i) \times dim(V_j))$ matrix. In case of Klein grading we have that elements of $End(V)_e$ have the following form:
\[\begin{bmatrix}
A_{11} & 0 & 0 & 0 \\
0 & A_{22} & 0 & 0 \\
0 & 0 & A_{33} & 0 \\
0 & 0 & 0 &  A_{44}\\
\end{bmatrix}\]
\newline elements of $End(V)_r$ have the following form:
\[\begin{bmatrix}
0 & A_{12} & 0 & 0 \\
A_{21} & 0 & 0 & 0 \\
0 & 0 & 0 & A_{34}\\
 0 & 0 & A_{43} &  0\\
\end{bmatrix}\]
\newline elements of $End(V)_s$ have the following form:
\[\begin{bmatrix}
0 & 0 & A_{13} & 0 \\
0 & 0 & 0 & A_{24} \\
A_{31} & 0 & 0 & 0 \\
0 & A_{42} & 0 &  0\\
\end{bmatrix}\]
\newline and  elements of $End(V)_t$ have the following form:
\[\begin{bmatrix}
0 & 0 & 0 & A_{41} \\
0 & 0 & A_{23} & 0 \\
0 & A_{32} & 0 & 0 \\
A_{41} & 0 & 0 & 0\\
\end{bmatrix}\]
In case of $\mathbb{Z}_2$-grading we have that  elements of $End(V)_0$ have the following form
\[\begin{bmatrix}
A_{11} & 0 \\
0 & A_{22}\\
\end{bmatrix}\]
and the elements of $End(V)_1$ have the following form
\[\begin{bmatrix}
0 & A_{12} \\
A_{21} &0\\
\end{bmatrix}\]

The graded algebra constructed this way will be denoted by $pl(V)$ and will be called general linear Lie Klein graded algebra or (general linear Lie superalgebra) of V. It plays the same role as the Lie algebra gl(V) does in the theory of Lie algebras. Now we can define

\begin{definition}
A graded representation of a Klein graded Lie algebra(resp. Lie superalgebra) L in a Klein graded(resp. $\mathbb{Z}_2$-graded) vector space V is an homomorphism of L into pl(V).
\end{definition}
\section{Derivations}\label{derivations}
\subsection{Derivations}

Denote with $\mathcal{K}$ the Klein group, then a derivation of degree s, s $\in \mathcal{K}$, of a Klein Graded algebra L is an endomorphism $D \in End_L$ with a property known as the Leibniz rule:
\begin{equation}\label{leibinz}
D(AB) = D(A)B + \theta(s,a)AD(B)
\end{equation} 
We denote by $der_sL\subset End_sL$ the space of all derivations of degree s, and we set $derL = \bigoplus_{s \in K} der_sL$. The space $derL \subset EndL$ is seen to be closed under the bracket:
\begin{proof}
Consider two derivations of degree s and t $d_1, d_2 \in derL$, we will show that the product $[d_1,d_2]$ is a derivation too. 
\begin{center}
 $[d_1,d_2](AB) = (d_1d_2 - \theta(s,t)d_2d_1)(AB) = d_1(d_2(A)) + \theta(s,d_2(A))d_2(A)d_1(B) +\theta(t,a)(d_1(A)d_2(B) + \theta(s,a)Ad_1d_2(B)) -\theta(s,t)(d_2d_1(A)B + \theta(t,d(A))d_1(A)d_2(B) + \theta(s,a)(d_2(A)d_1(B) + \theta(t,a)Ad_2d_1(B)))$
= $[d_1,d_2](A)B  + \theta(st,a)A[d_1,d_2](B)$.
\end{center}
so, it is a derivation of degree st.
\end{proof}
In other words it is a subalgebra of $EndL$; it is called the algebra of derivations of L. Every element of derL is called derivation of L.
\begin{example}
Let L be a Klein graded Lie Algebra. It follows from the jacobi identity that $ad_A:B \to [A,B]$ is a derivation of L. These derivations are called inner; they form an ideal of derL, because$ [D, ad_A] = ad_{DA}$ for all D $\in derL$.
\begin{proof}
We will show that ad is a derivation using the jacobi identity.
\begin{center}
$0=\theta(z,x)[X,[Y,Z]] + \theta(y,z)[Z,[X,Y]] + \theta(x,y)[Y,[Z,X]]\Leftrightarrow$
\newline $-\theta(z,x)ad_X[Y,Z] = \theta(y,z)[Z,ad_X Y] + \theta(x,y)[Y,-\theta(z,x)ad_X Z]\Leftrightarrow$
\newline $-\theta(z,x)ad_X[Y,Z] = -\theta(y,z)\theta(z,ad_XY)[ad_XY,Z]-\theta(x,y)\theta(z,x)[Y,ad_XZ]\Leftrightarrow$
\newline $ad_X[Y,Z] = \theta(z,x)\theta(y,z)\theta(z,ad_Y)[ad_XY,Z] +\theta(x,y)[Y,ad_XZ]$
\end{center}
It is enough now to show that $\theta(z,x)\theta(y,z)\theta(z,ad_Y) =1$. This is true: 
\begin{center}
$\theta(z,x)\theta(y,z)\theta(z,ad_Y) = \theta(z,XYad_XY)$, and the element $XYad_XY \in L_e.$ We recall that $\theta(a,e) =1, \forall a \in K$
\end{center}
The same holds for a Lie superalgebra and the proof is pretty much the same.
\end{proof}
\end{example}

\begin{remark}

The adjoint representation ad of a Lie superalgebra $L=L_0 \bigoplus L_1$ induces a representation of the Lie algebra $L_0$ in the odd subspace $L_1$. This representation is called the adjoint representation of $L_0$ in $L_1$. Using this representation we can give a new description of Lie superalgebras, as follows.

Let $L=L_0 \bigoplus L_1$ be a Lie superalgebra. For convenience we shall write for the moment $\tilde{Q}$ instead of adQ, Q $\in L_0$, hence
\begin{equation}
\tilde{Q}(U) = [Q,U]=-[U,Q] \text{ for all } Q \in L_0, U \in L_1
\end{equation}
The Lie superalgebra L is uniquely fixed if we are given the Lie algebra $L_0$, the representation $Q \to \tilde{Q}$ of $L_0$ in $L_1$ and the symmetric bilinear product mapping
\begin{equation}
P : L_1 \times L_1 \to L_0
\end{equation}
The graded jacobi identity impLies that P is $L_0$ invariant, i.e. that
\begin{equation}
[Q,P(U,V)] = P(\tilde{Q}(U),V) + P(U,\tilde{Q}(V))
\end{equation}
\begin{center}
for all Q $\in L_0$, U,V $\in L_1$,
\end{center}
and that, furthermore,
\begin{equation}\label{jacobilike}
\tilde{P(U,V)}(W) + \tilde{P(V,W)}(U) + \tilde{P(W,U)}(V) = 0
\end{equation}
\begin{center}
for all W, U, V $\in L_1$.
\end{center}
Conversely, let $L_0$ be a Lie superalgebra and let $Q \to \tilde{Q}$ be a representation of $L_0$ in some vector space $L_1$. Suppose we are given a summetric bilinear mapping P of $L_1 \times L_1$ to $L_0$. We define a multiplication [ , ] on the $\mathbb{Z}_2$-graded vector space $L_0 \bigoplus L_1$ by
\begin{equation}
[Q,R] = [Q,R] \text{   if Q,R } \in L_0
\end{equation}
\begin{equation}
[Q,U] = -[U,Q] = \tilde{Q}(U) \text{   if Q } \in L_0 , U \in L_1
\end{equation}
\begin{equation}
[U,V] = P(U,V) \text{   if U,V } \in L_1
\end{equation}
With this multiplication $L_0 \bigoplus L_1$ is a Lie superalgebra if and only if P is $L_0$-invariant and if \eqref{jacobilike} is fulfilled. In a sense, therefore, a Lie superalgebra $L=L_0 \bigoplus L_1$ is some "superstructure" to be built over the Lie algebra $L_0$.
\end{remark}
One will wonder if we can say the same for a Klein graded Lie algebra, and we will show that it is more complicated but possible.
\begin{proposition}\label{53prop}
Let $L = L_e \bigoplus L_r \bigoplus L_s \bigoplus L_t$ be a Klein graded Lie algebra. We will denote $ad^{i}_{Q}$ the representation of $L_e$ in $L_i, i \in \{r,s,t\}$, hence
\begin{equation}
ad^{i}_{Q}(U) = [Q,U] = -[U,Q] \text{ for all } Q \in L_e, U \in L_i
\end{equation}
The Klein graded Lie algebra L is uniquely fixed if we are given the Lie algebra $L_e$, the representations $ad^{r}_{Q}$, $ad^{s}_{Q}$,$ad^{t}_{Q}$ of $L_e$ in $L_r$,$L_s$,$L_t$ respectively and the symmetric bilinear product mappings
\begin{equation}
P_{xy} : L_x \times L_y \to L_{xy} \text{ for all } x,y \in \mathcal{K} -\{e\}
\end{equation}
The graded Jacobi identity impLies that these mappings are $L_e$-invariant, i.e. that
\begin{equation}
[Q,P_{ij}(I,J)] = P_{ij}(ad^{i}_{Q}(I),J) + P_{ij}(I,ad^{j}_{Q}(J))
\end{equation}
\begin{center}
for all Q $\in L_e$, I $\in L_i$, J $\in L_j$, i,j $\in$ \{r,s,t\}
\end{center}
or
\begin{equation}\label{kleinjacobilike}
0 = \theta(a,c)P_{bca}(A,P_{bc}(B,C)) + \theta(a,b)P_{cab}(B,P_{ac}(A,C)) + \theta(b,c)P_{abc}(C,P_{ab}(A,B))
\end{equation}
\begin{center}
for all A $\in L_a$, B $\in L_b$, C $\in L_c$, a,b,c $\in \mathcal{K} -\{e\}$.
\end{center}
Conversely, let $L_e$ be a Lie algebra and let $ad^{r}_{Q}$, $ad^{s}_{Q}$,$ad^{t}_{Q}$ be three representations of $L_e$ in some vector spaces $L_r$,$L_s$,$L_t$ respectively. Suppose we are given 6 symmetric bilinear mappings $P_{xy}$ of $ L_x \times L_y$ into $ L_{xy}$ and a color $\theta$ defined on the Klein group $\mathcal{K}$. We define multiplication [ , ] on the Klein-graded vector space $ L_e  \bigoplus L_r \bigoplus L_s \bigoplus L_t $ by
\begin{equation}
[Q,R] = [Q,R] \text{   if Q,R } \in L_e.
\end{equation}
\begin{equation}
[Q,A] = -[A,Q] = ad^{a}_{Q}(A) \text{   if Q } \in L_0 , U \in L_a, a \in \mathcal{K} -\{e\}.
\end{equation}
\begin{equation}
[A,B] = P_{ab}(A,B) \text{   if A} \in L_a, B \in L_b, a,b  \in \mathcal{K} -\{e\}.
\end{equation}
With this multiplication $ L_e  \oplus L_r \oplus L_s \oplus L_t $ is a Klein graded Lie algebra if and only if $P_{xy}$ are $L_e$-invariant and if \eqref{kleinjacobilike} is fulfilled. 
\end{proposition}

\subsection{Solvability and nilpotency}
Regarding the solvability and nilpotency definitions, they are the same as for a Lie algebra. Engel's theorem and its direct consequences remain valid, and the proof is the same as for Lie algebra \cite{humphreys}.

\begin{definition}
A Lie superalgebra(resp. a Klein graded Lie algebra) L is nilpotent if there exists a  positive integer k such that
\[ [L,L_k] = 0 \; \; \; ; \; \; \; L_k = [L,L_{k-1}]\].
\end{definition}

\begin{proposition}\label{nil1}
Let $V\neq \{0\}$ be a $Z_2$ graded vector space(resp. a Klein graded vector space) and let $L$ be a graded subalgebra $pl(V)$ such every homogeneous element of $L$ is nilpotent. Then there exists a non-zero element $v \in V$ such that $X(v) = 0$ for all $X \in L$.
\end{proposition}

\begin{corollary}
A Lie superalgebra(resp. a Klein graded Lie algebra) $L$ is nilpotent if and only if $ad_LX$ is nilpotent for every homogeneous element $X$ of $L$.
\end{corollary}

A proof for the proposition \ref{nil1} and the corollary can be constructed by using the proof for the same results in \cite{zhang2} for the case of $L$ being a Lie color algebra.

\begin{corollary}
The Lie superalgebra(resp. a Klein graded Lie algebra) $L$ in proposition \ref{nil1} is nilpotent.
\end{corollary}

\begin{definition}
A Lie superalgebra(resp. a Klein graded Lie algebra) L is solvable if there exists a  positive integer k such that
\[ [L^k,L^k] = 0 \; \; \; ; \; \; \; L^k = [L^{k-1},L^{k-1}]\]
\end{definition}

On the other hand, Lie's theorem does not necessarily hold for a solvable Lie superalgeba. Moreover, Zhang in \cite{zhang2} shown that it hold for a special case of Lie color algebras (for those which the group of the gradation is torsion free) and an example is given to show that it fails on the case of a color Lie algebra with a torsion group. Kac \cite{kac2} has investigated the finite-dimensional irreducible graded representations of solvable Lie superalgebras and the following are some of his results.

\begin{proposition}
A Lie superalgebra $L$ is solvable if and only if its Lie $L_0$ is solvable.
\end{proposition}

\begin{proposition}
Suppose that the field $K$ is algebraically closed. All the finite-dimensional irreducible graded representations of a solvable Lie superalgebra $L$ are one-dimensional if and only if $[L_1,L_1] \subset [L_0,L_0]$.
\end{proposition}

\subsection{Modules}\label{derivationsmodules}
Let L be a Klein graded Lie algebra. A L-module M is a Klein graded vector space
\[M = \bigoplus_{k \in K} M_k\]
where K is the Klein group, such that $L_g.M_h \subset M_{g+h}$ for every g, h $\in$ K and 
\[[x,y].m = x.(y.m) - \theta(x,y)y.(x.m)\]
for every $m \in M , x,y \in L$. Note that a homomorphism $r:L \to$ End(M) of Klein graded Lie algebras
defines a L-module structure on M and vice vera.
\newline
Let V and W be L-modules, and f $\in Hom(V,W)_k, k \in \mathcal{K}$ such that f(xm) = $\theta(k,x)x.f(m)$, for every $x \in L, m \in M$, then f is called a color homomorphism of L-modules. If $\theta(k,x) =e$, then f is called a homomorphism of L-modules.
Now we can give a more abstract definition for derivations, and we will prove some propositions introduced in \cite{Erdos01}.
\begin{definition}\label{derivationsonmodules}
Let V be a L-module. A linear mapping f: L $\to$ V is called a derivation if
\begin{center}
f([x,y]) = $\theta(f,x)x.f(y) - \theta(fx,y)y.f(x)$,
\end{center}
for every x,y $\in$ L. The derivations of the form $x \to x.u (u \in V)$ are called inner(as expected). We say that a derivation f has degree k $ (der(f)=k)$ if 
\[f \neq 0\]
and 
\[f(L_h) \subset V_{k+h}\]
 for all h $\in K$.
\end{definition}
This is compatible with the previous definition of derivation. 
This is true due to the fact that  $\theta(f+ x, y)y.f(x)= f(x).y$.
\newline
We denote with Der(L,V) and Inn(L,V) the spaces of derivations and inner derivations respectively, and we write $Der(L,V)_k = \{f \in Der(LV): deg(f)=k\} \cup \{0\}$. Note that $H^1(L,V) = Der(L,V)/Inn(L,V)$ is the first cohomology group of L with coefficients in V.

\begin{proposition}\label{divideprop}
Let V be a L-module. Then Der(L,V) = $\bigoplus_{k \in K} Der(L,V)_k$.
\end{proposition}
\begin{proof}
For each element k in klein group, let $r_k : L \to L_k$ and $p_k:V\to V_k$ denote the canonical projections. According to our general assumption there is a finite subset S$\subset$L generating L. Let $f:L \to V$ be a derivation. Then there are finite sets Q,R $\subset$ K such that
\begin{equation}
S \subset \bigoplus_{g \in Q}L_g,
\end{equation}
and
\begin{equation}
f(S) \subset \bigoplus_{g \in R}V_g,
\end{equation}
For $g \in K$, we set $f_g = \sum_{h \in K} p_{k+h} f r_h$. If $x_h \in L_h$ and $x_k \in L_k$, then we have:
\begin{center}
$f_g([x_h,x_k]) = p_{g+h+k}f([x_h,x_k]) = p_{g+h+k}(\theta(f,h)x_hf(x_k) - \theta(f + h,k)x_kf(x_h)) = \theta(f,h)x_hp_{g+k}(f(x_k)) - \theta(f+h,k)x_kp_{g+h}(f(x_h)) =  \theta(f,h)x_hf_g(x_k) - \theta(f+h,k)x_kf(x_h)$,
\end{center}
it shows that $f_g$ is contained in $Der(L,V)_g$
\newline
Let T=\{g-h : g $\in$ R, h $\in$ Q\}. Then T is finite and we obtain(keep in mind (4) and (5)), for y $\in$ S
\begin{align*}
f(y) &= \sum_{g \in R} p_gf(y) = \sum_{g \in R}\sum_{h \in Q}p_gfr_h(y)
\\ &= \sum_{h \in Q}(\sum_{g \in R}p_{g-h} + hfr_h(y)) =\sum_{h \in Q}(\sum_{q \in T} p_q + hfp_h(y))
\\ &=\sum_{q \in T}(\sum_{h \in Q}p_{q+h}fr_h(y)) =\sum_{q \in T}\sum_{h \in G}p_{q+h}fp_h(y)  = \sum_{q \in T}f_q(y).
\end{align*}
This shows that the derivation f and $ \sum_{q \in T}f_q$ coincide on S. As S generates L, we obtain f = $ \sum_{q \in T}f_q(y)$. This proves the assertion.
\end{proof}

\begin{proposition}
Let V be a L-module such that

\begin{enumerate}[{ }i{.}]
\item $H^1(L_0,V_g) = 0, \forall g \in K \backslash \{0\}$
\item $Hom_{L_0}(L_g, V_) = 0, g \neq h.$
\end{enumerate}
Then Der(L,V) = $Der(L,V)_0$ + Inn(L,V).
\end{proposition}
\begin{proof}
Let $f : L \to V$ be a color derivation. According to proposition \ref{divideprop}, we can decompose f into its homogeneous components: f = $\sum_{g \in K}f_g , f_g \in Der(L,V)_g$. Suppose that g $\neq$ 0. Then $f_g | L_0$ is a color derivation from $L_0$ into the $L_0-module V_g$. Using (a) $f_g |L_0$ is inner, i.e., there is $u_g \in V_g$ such that $f_g(x) = x.v_g$, for all $x \in L_0$. Consider $\psi_g:L \to V, \psi_g(x)=f_g(x) - \theta(g,x)x.u_g, \forall x \in L$. Then $\psi_g$ is a color derivation of degree g, which vanishes on $L_0$. Hence, for every $x_0 \in L_0, x\in L_h$,
\begin{center}
$\psi_g([x_0,x]) = \theta(g,0)x_0\psi_g(x) - \theta(g +0,h)x.\psi_g(x_0) = x_0.\psi_g(x).$
\end{center}
Thus, $\psi_g$ is a color homomorphism of $L_0$-modules and condition (b) entails the vanishing of $\psi_g$ on $L_h$ for every $h \in K$. Consequently, $f_g \in Inn(L,V)$, proving the desired.
\end{proof}

\begin{remark}
Let L be a Klein graded Lie algebra and V a L-module. Supose that $\tau : L \to V$ is a color homomorphism of L-modules of degree r, i.e., $\tau(L_g) \subset V_g+r$, for all g $\in K$. If $\sigma : G \to \mathbb{C}$ is Z-linear, then the linear mapping $\phi_{\sigma, \tau} : L \to V$ given by  $\phi_{\sigma, \tau}(x_g) = \sigma(g)\tau(x_g)$, for every $x_g \in L_g, g\in K$ is a derivation of degree r. 
\end{remark}
\begin{proof}
For every $x_g \in L_g, x_h \in L_h$, we have
\begin{align*}
\phi_{\sigma,\tau}([x_g,x_h]) &= \sigma(g+h)\tau([x_g,x_h]) = (\sigma(g) + \sigma(h))\tau([x_g,x_h])
\\ &= \sigma(g)\tau([x_g,x_h]) + \sigma(h)\tau([x_g,x_h])
\\ &=-\theta(g,h)\theta(r,h)\sigma(g)x_h\tau(x_g) + \sigma(h)\theta(r,g)x_g\tau(x_h)
\\ &=\theta(r,g)x_g\phi_{\sigma,\tau}(x_h) - \theta(g+r,h)x_h\phi_{\sigma,\tau}(x_g).
\end{align*}
This is the desired result.
\end{proof}
Let $J \subset L$ be an ideal, V be a L-module and consider the L-submodule $V^J = \{u \in V | x.u = 0 \forall x \in J\}$ of V.
\begin{remark}
Let A $\subset$ L be a Klein graded subalgebra, J $\triangleleft $ L a Klein graded ideal. 
If $L = A\bigoplus J$, then every derivation f$:A \to V^L$ can be extended to a derivation $f^* : L \to V$ be setting $f^*(J) =0$.
\end{remark}
\begin{proof}
Let $f^* : L \to V$ such that
$$ f^*(x)= \begin{cases}
f(x) & \text{, }  x \in A \\
x  &  \text{, } x \in J. 
\end{cases}$$
Then $f^*([a,x]) = 0, a \in a, x \in j$. But
\begin{center}
$\theta(f^*,a)a.f^*(x) - \theta(f^* + a,x)x.f^*(a) = -\theta(f^* +a,x)x.f^*(a) = \theta(f^* +a,x)x.f(a)$.
\end{center}
Sinse f(a) $\in V^J$, then x.f(a) =0, i.e., $f^*([a,x]) = \theta(f^*,a)a.f^*(x) - \theta(f^* +a, x)x.f^*(a)$.
Hence $f^*$ is a color derivation and $f^*(J) =0$.
\end{proof}

\section{The Enveloping Algebra}\label{envelopingalgebra}

In this paragraph we shall introduce the enveloping algebra of a Lie superalgebra and the envoping algebra of a Klein graded Lie algebra. As to be expected from the Lie algebra case the enveloping algebra turns out to be a very useful tool for the theory of Lie superalgebras(or  a Klein graded Lie algebra) and their representations.

\subsection{The tensor algebra}

Let $V$ be a vector space over a field $K$. For any nonnegative integer $k$, we define the kth tensor power of $V$ to be the tensor product of V with itself $k$ times:
\[T_k(V) = V \otimes V \otimes \dots \otimes V.\]
That is, $T_kV$ consists of all tensors on $V$ of rank $k$. By convention $T_0V$ is the ground field $K$ (as a one-dimensional vector space over itself). We construct $T(V)$ as the direct sum of $T_k(V)$:
\[T(V) = \bigoplus_{k=0}^{\infty} T_k(V)\]
The multiplication in $T(V)$ is determined by the canonical isomorphism
\[T_k(V) \otimes T_l(V) \to T_{k+l}(V)\]
given by the tensor product, which is then extended by linearity to all of $T(V)$. This multiplication rule implies that the tensor algebra $T(V)$ is naturally a graded algebra with $T_k(V)$ serving as the grade-k subspace. This grading can be extended to a Z grading by appending subspaces $T_k(V) = \{0\}$ for negative integers k.

\subsection{Definition for Lie superalgebras}\label{72}
Let $L=L_0 \bigoplus L_1$ be a Lie superalgebra and let T(L) be the tensor algrbra of the vector space L. The $Z_2$ gradation of L induces a $Z_2$ gradation T(L) such that the canonical injection $L \to T(L)$ is an even linear mapping and that T(L) is an associative superalgebra. The equations below give us a way to construct the $Z_2$ gradation of T(L) in terms of L

\begin{equation*}
 ( L \otimes L)_t = \bigoplus_{x+y=t} L_x \otimes L_ y
\end{equation*}
\begin{equation*}
 (L_{x_1} \otimes L_{ y_1})(L_{x_2} \otimes L_ {y_2}) =  L_{x_1 + x_2} \otimes L_{y_1 + y_2} 
\end{equation*}
We consider the two-sided ideal J of T(L) which is generated by the elements
\begin{equation*}
 A\otimes B - (-1)^{ab}  B \otimes A  + [A,B] 
\end{equation*}
\begin{center}
with A $\in L_a$, B $\in L_b$ ; a,b $\in Z_2$
\end{center}
Evidently these elements are homogeneous (of degree a+b), hence J is a graded ideal. Therefore, if we define
\begin{equation*}
U(L) = T(L) / J
\end{equation*}
it follws that U(L) is an assosiative superalgebra; this algebra is called enveloping algebra of L. By composing the canonical injection $L \to T(L)$ with the canonical mapping $T(V) \to U(L)$ we obtain the canonical even linear mapping
\[\sigma : L \to U(L) \]

which satisfies the following condition:
\begin{equation}\label{sigmaofuniversal}
\sigma([A,B]) = \sigma(A)\sigma(B) - (-1)^{ab}\sigma(B)\sigma(A)
\end{equation}
\[\text{for all } A \in L_a \; , \; B \in L_b \; ; \; a,b \in Z_2.\]
Every element of $U(L)$ is a linear combination of products of the form
\[\sigma(A_1) \dots \sigma(A_r) \; \text{ with } A_i \in L_i \; , \; a_i \in Z_2 \; ; \; 1 \leq i \leq r\]
(for $r=0$ we define this product to be equal to 1) ; this product is a homogeneous element of $U(L)$ of degree $a_1 + \dots a_r$. The pair $(U(L), \sigma)$ is characterized by \ref{sigmaofuniversal} and by the following universal property:
\begin{proposition}\label{maponassociative}
Let S be an associative algebra with unit element and let g be a linear mapping of $L$ into $S$ such that
\begin{equation*}
g([A,B]) = g(A)g(B) - (-1)^{ab}g(B)g(A)
\end{equation*}
\[\text{for all } A \in L_a \; , \; B \in L_b \; ; \; a,b \in Z_2.\]
Then there exists a unique homomorphism $\bar{g}$ of the algebra $U(L)$ into the algebra $S$ such that
\[g = \bar{g} \circ \sigma \; \; , \; \; \bar{g}(1) =1.\]
\end{proposition}

\begin{proof}
Let us use a diagram to make more clear
\[
\begin{tikzcd}
&L \arrow{d}{g}\arrow{r}{\sigma}
&U(L) \arrow{ld}{\bar{g}} \\
&S 
\end{tikzcd}
\]
Now recall that every element of $U(L)$ is a linear combination of products of the form
\[\sigma(A_1) \dots \sigma(A_r) \; \text{ with } A_i \in L_i \; , \; a_i \in Z_2 \; ; \; 1 \leq i \leq r.\]
We define $\bar{g}$ as
\[\bar{g}(\sigma(A_1) \dots \sigma(A_r) ) = g(A_1) \dots g(A_r)\]
\[\text{ with } A_i \in L_i \; , \; a_i \in Z_2 \; ; \; 1 \leq i \leq r.\]
set $\bar{g}(1)=1$ and linear extend to all elements of $U(L)$. $\bar{g}$ is a homomorphism of associative algebras:
\[\bar{g}(1)=1\]
\[\bar{g}(rx) = r\bar{g}(x) \text{ for all x in } U(L)\]
\[\bar{g}(x + y) = \bar{g}(x + y) \text{ for all x,y in } U(L)\]
\begin{align*}
& \bar{g}(\sigma(A_1) \dots \sigma(A_r)\sigma(B_1) \dots \sigma(B_k)) = 
\\ & g(A_1) \dots g(A_r)g(B_1) \dots g(B_k) =
\\ & \bar{g}(\sigma(A_1) \dots \sigma(A_r))\bar{g}(\sigma(B_1) \dots \sigma(B_k)) 
\\ & \text{ for all } A_i,B_j \in  L \; ; 1 \leq i \leq r \; , 1 \leq j \leq k 
\end{align*}
Hence, $\bar{g}$ is a homomorphism of these associative algebras and
\[g = \bar{g} \circ \sigma \; \; , \; \; \bar{g}(1) =1.\]
Suppose that there exists another homomorphism $f : U(L) \to S$ with 
\[g = f \circ \sigma \; \; , \; \; f(1) =1.\]
then on every element of the form
\[\sigma(A_1) \dots \sigma(A_r) \; \text{ with } A_i \in L_i \; , \; a_i \in Z_2 \; ; \; 1 \leq i \leq r.\]
we have that 
\[f(\sigma(A_1) \dots \sigma(A_r)) =f(\sigma(A_1)) \dots f(\sigma(A_r))= g(A_1) \dots g(A_r) = \bar{g}(A_1) \dots \bar{g}(A_r)\]
Hence, $f = \bar{g}$.
\end{proof}

\subsection{Definition for Klein graded Lie algebras}\label{73}
Now let L be a Klein graded Lie algebra and let T(L) be the tensor algebra of the vector space L. Again, the $\mathcal{K}$-gradation of L induces a $\mathcal{K}$-gradation on T(L). The equations below give us a way to construct the Klein gradation of T(L) in terms of L
\begin{equation*}
 ( L \otimes L)_t = \bigoplus_{xy=t} L_x \otimes L_ y
\end{equation*}
\begin{equation*}
 (L_{x_1} \otimes L_{ y_1})(L_{x_2} \otimes L_ {y_2}) =  L_{x_1 x_2} \otimes L_{y_1 y_2} 
\end{equation*}
We consider the two-sided ideal J of T(L) which is generated by the elements
\begin{equation*}
 A\otimes B - \theta(a,b)  B \otimes A  + [A,B] 
\end{equation*}
\begin{center}
with A $\in L_a$, B $\in L_b$ ; a,b $\in  \mathcal{K}$
\end{center}
Evidently these elemens are homogeneous (of degree ab), hence J is a graded ideal.
Therefore, if we define
\begin{equation*}
U(L) = T(L) / J
\end{equation*}
it follws that U(L) is an assosiative Klein graded algebra; this algebra is called enveloping algebra of L. By composing the canonical injection $L \to T(L)$ with the canonical mapping $T(V) \to U(L)$ we obtain the canonical even linear mapping
\[\sigma : L \to U(L) \]

which satisfies the following condition:
\begin{equation}\label{sigmaofuniversal2}
\sigma([A,B]) = \sigma(A)\sigma(B) - \theta(a,b)\sigma(B)\sigma(A)
\end{equation}
\[\text{for all } A \in L_a \; , \; B \in L_b \; ; \; a,b \in\mathcal{K}.\]
Every element of $U(L)$ is a linear combination of products of the form
\[\sigma(A_1) \dots \sigma(A_r) \; \text{ with } A_i \in L_i \; , \; a_i \in \mathcal{K} \; ; \; 1 \leq i \leq r\]
(for $r=0$ we define this product to be equal to 1) ; this product is a homogeneous element of $U(L)$ of degree $a_1 + \dots a_r$. The pair $(U(L), \sigma)$ is characterized by \ref{sigmaofuniversal} and by the following universal property:
\begin{proposition}
Let S be an associative algebra with unit element and let g be a linear mapping of $L$ into $S$ such that
\begin{equation*}
g([A,B]) = g(A)g(B) - \theta(a,b)g(B)g(A)
\end{equation*}
\[\text{for all } A \in L_a \; , \; B \in L_b \; ; \; a,b \in \mathcal{K}.\]
Then there exists a unique homomorphism $\bar{g}$ of the algebra $U(L)$ into the algebra $S$ such that
\[g = \bar{g} \circ \sigma \; \; , \; \; \bar{g}(1) =1.\]
\end{proposition}

\subsection{The supersymmetric algebra of a graded vector space} 
Let $V=V_0 \oplus V_1$ be a $Z_2$ graded vector space and let $T(V)$ be the tensor  algebra of $V$. Then $T(V)$ is a $Z$ graded algebra,
\[T(V) = \bigoplus_{n \in Z} T_n(V)\]
where $T_m(V) =\{0\}$ if $m \leq -1$ and where $T_n(V)$ is the vector space of tensors of order $n$ if $n\geq 0$. If $T(V)$ is equipped with the $Z_2$ gradation of inherited $V$ then all the $T_n(V)$ are $Z_2$ graded subspaces of $T(V)$; hence $T(V)$ is an associative $Z$ graded superalgebra. Now let $I$ be the two-sided ideal of $T(V)$ which is generated by the tensors of the form
\[A \otimes B -(-1)^{ab}B \otimes A\]
\[\text{with } a\in L_a, b \in L_b \; ;a,b \in Z_2\]
These tensors are homogeneous both with respect to the $Z_2$ gradation and with respect to the $Z$ gradation of $T(V)$. Hence 
\[S(V) = T(V) / I \]
is an associative $Z$ graded superalgebra which is called the supersymmetric algebra of the $Z_2$ graded vector space $V$. If we define $S_n(L)$ the images of $T_n(L)$ under the cannonical mapping, then we can write
\[S(V) = \bigoplus_{n \in Z} S_n(V)\]
\begin{remark}
Notice that if V is equiped with the trivial bracket
\[ [V,V] = \{0\}\]
then $V$ is an (abelian) Lie superalgebra and $S(V)$ is nothing else but the enveloping algebra of this Lie superalgebra. Recall that the universal algebra is defined as $U(L) = T(L) / J$, where J is the two-sided ideal J of T(L) which is generated by the elements
\begin{equation*}
 A\otimes B - (-1)^{ab}  B \otimes A  + [A,B] 
\end{equation*}
\begin{center}
with A $\in L_a$, B $\in L_b$ ; a,b $\in Z_2$
\end{center}
Hence, $U(L) = S(L)$.
\end{remark}
It is easy now to define the Klein symmetric algebra of a Klein graded vector space. We set $I$ the two-sided ideal of the Tensor algebra which is generated by the tensors of the form
\[A \otimes B -\theta(a,b)B \otimes A\]
\[\text{with } a\in L_a, b \in L_b \; ;a,b \in \mathcal{K}\]
\begin{remark}
If V is equiped with the trivial bracket
\[ [V,V] = \{0\}\]
then $V$ is an (abelian) Klein graded Lie algebra and $S(V)$ is nothing else but the enveloping algebra of this Klein graded Lie algebra.
\end{remark}

\subsection{Filtration of the enveloping algebra}
The following definitions and propositions are the same for both a Lie superalgebra and a Klein graded Lie algebra. Let $L$ be a Lie superalgebra(respectively a Klein graded Lie algebra) and let $T(V)$ the tensor algebra of the vector space $L$. In the preceding section we have seen that $T(L)$ has a  natural structure of a $Z$ graded superalgebra(resp. $Z$ graded Klein algebra)
\[T(L) = \bigoplus_{n \in Z}T_n(L)\]
We define for all $n \in Z$
\[T^n(L) = \bigoplus_{m \leq n} T_m(L) \; \]
the $T^n(L)$ are $Z_2$ graded(resp. Klein graded)  subspaces of $T(L)$.
\begin{remark}
Here we have to note that 
\[T^n(L) \otimes T^m(L) \subset T^{n+m}(L)\]
which is true due to the multiplication of $T(L)$ and the definition of $T^n(L)$.
\end{remark}
Now let $U(L)$ be the enveloping algebra of $L$ and let $\sigma : L \to U(L)$ be the canonical mapping. Let $U^n(L)$ be the image of $T^n(L)$ under the canonical mapping $ \pi:T(L) \to U(L)$. 

\begin{proposition}
It is easy to verify the following statements

\begin{enumerate}
\item{
If $n \geq 0$ the subspace $U^n(L)$ of $U(L)$ is generated by the products of the form $\sigma(A_1)\dots \sigma(A_m)$ with $0 \leq m \leq n $ and $A_1, \dots, A_m \in L$.
}
\item{
The $U^n(L)$ are $Z_2$ (resp. Klein) graded subspaces of $U(L)$.
}
\item{
$U^n(L) \subset U^m(L) \; \text{ if } \; n \leq m$
}
\item{
$U^n(L)= \{0\} \; \text{ if } \; n \leq -1$
}
\item{$U^0(L) = K  \cdot 1$}
\item{
$\bigcup_{n \geq 0} U^n(L) = U(L)$
}
\item{
$U^n(L) U^m(L) \subset U^{n+m}(L)\; \text{ for all } \; n \in Z$
}
\end{enumerate}
\end{proposition}
\begin{proof}
\begin{enumerate}
\item{
$U^n(L)$ contains tensors with degree up to $n$. We can prove this statement by induction to n. For $n=0$ the statement is obvious because $U^0(L)$ is the ground field $K$ and the product is equal to 1 as it was defined in a previous section. Assume that it holds for $n$ and we will prove it for $n+1$. Note that 
\[U^{n+1}(L) = U^n(L) \oplus T_{n+1}/J\]
where J the ideal of $T(L)$ as defined in the definition of the universal algebra. Then $U^n(L)$ is generated by products of the form 
\[\sigma(A_1)\dots \sigma(A_m)\]
\[\text{ with } 0 \leq m \leq n  \text{ and } A_1, \dots, A_m \in L\]
 and $ T_{n+1}/J$ is generated by products of the form 
\[\sigma(B_1)\dots \sigma(B_{n+1})\]
\[\text{ with } B_1, \dots, B_{n+1} \in L\]
Hence, $U^{n+1}(L)$ is generated by products of the form
\[\sigma(A_1)\dots \sigma(A_m)\]
\[\text{ with } 0 \leq m \leq n+1  \text{ and } A_1, \dots, A_m \in L\]
}
\item{
As we described in sections \ref{72}, \ref{73} all $T_n(L)$ have a $Z_2$ (resp. Klein) gradation which is inherited to $U^n(L)$.
}
\item{
If $n \leq m$ then 
\[T^n(L) = \bigoplus_{k \leq n} T_k(L) \subset \bigoplus_{k \leq m} T_k(L)\]
Hence, $U^n(L) \subset U^m(L)$ too.
}
\item{
$T^n(L) = 0$ for $n \leq -1$ by definition, hence $U^n(L) = 0$ for $n \leq -1$ too.
}
\item{
$T^0(L) = K$.
}
\item{
We have that
\[\cup_{n \geq 0}T^n(L) = T(L)\]
by definition. So through the cannonical mapping $\pi$ we get the desired result.
}
\item{
We know that
\[T_n(L) T_m(L) \subset T_{n+m}(L)\; \text{ for all } \; n \in Z\]
by the definition of the multiplication on $T(L)$. Hence
\[T^n(L) T^m(L) \subset T^{n+m}(L)\; \text{ for all } \; n \in Z\]
and through the cannonical mapping $\pi$ we get the desired result.
}
\end{enumerate}
\end{proof}
The family $(U^n(L))_{n \in Z}$ is called the canonical filtration of the enveloping algebra $U(L)$.
Now let 
\[G_n(L) = U^n(L)/U^{n-1}(L)\]
and
\[G(L) = \bigoplus_{n \in Z} G_n(L)\]
Moreover, the multiplication in $U(L)$ define a bilinear map \[G_n(L) \times G_m(L) \to G_{n+m}(L)\]
and it can be extended to a bilinear map
\[G \times G \to G\]
which makes $G$ a $Z$ graded associative algebra with unity. $G(L)$ has a natural $Z_2$ (resp. Klein) gradation inherited from $U(L)$ too. Therefore, G(L) is a $Z$ graded assiciative superalgebra (resp. Klein graded algebra).

Denote the composition of the canonical mapping $T_n\to U^n(L) \to G_n(L)$ by $\phi_n$ and let $\phi : T(L) \to G(L)$ be the linear mapping defined by the family $(\phi_n)_{n \in Z}$. Then $\phi$ is surjective homomorphism of $Z$ graded superalgebras (resp. Klein graded algebras) which vanishes on all tensors of the form
\[A \otimes B - (-1)^{ab} B\otimes A\]
\[\text{with \; } A\in L_a \;,\;B \in L_b \; ; \; a,b \in Z_2\]
Respectively
\[A \otimes B - \theta(a,b) B\otimes A\]
\[\text{with \; } A\in L_a \;,\;B \in L_b \; ; \; a,b \in \mathcal{K}\]

Consequently, $\phi$ defines a homomorphism $\bar{\phi} : S(L) \to G(L)$.
\begin{lemma}
$\phi : T(L) \to G(L)$ is an algebra homomorphism. Moreover, $\phi(I) = 0$ ( where $I$ the ideal of $T(L)$ that form the symmetric algebra), so $\phi$ induces a homomorphism $\bar{\phi}$ of $S(L)$ into $G(L)$.
\end{lemma}
\begin{proof}
Let $x \in T_m , y \in T_p$ be homogeneous tensors. By definition of the product in $G(L)$, $\phi(xy) = \phi(x)\phi(y)$, so it follows that $\phi$ is multiplicative on $T(L)$. Let 
\[x\otimes y - (-1)^{ab}y\otimes x\]
\[x \in L_a, y \in L_b ; a,b \in Z_2\]
(respectively for the Klein version)
\[x\otimes y - \theta(a,b)y\otimes x\]
\[x \in L_a, y \in L_b ; a,b \in \mathcal{K}\]
 be a typical generator of $I$. Then 
\[\pi(x\otimes y - (-1)^{ab}y\otimes x) \in U^2(L)\]
\[\text{Respectively: }\pi(x\otimes y - \theta(a,b)y\otimes x) \in U^2(L)\]
by definition. On the other hand we have that
\[\pi(x\otimes y - (-1)^{ab}y\otimes x) = \pi([x,y]) \in U^1(L)\]
\[\text{Respectively: }\pi(x\otimes y - \theta(a,b)y\otimes x) = \pi([x,y]) \in U^1(L)\]
therefore, 
\[\phi(x\otimes y - (-1)^{ab}y\otimes x)  \in U^1(L)/U^1(L) = 0\]
\[\text{Respectively: }\pi(x\otimes y - \theta(a,b)y\otimes x) \in U^1(L)/U^1(L) = 0\]
It follows that $I \subset kef\phi$.
\end{proof}

\subsection{Poincar\'e-Birkhoff-Witt}
In order to prove this theorem, we need two more definitions and a lemma. 
\begin{definition}
Let L be a Lie superalgebra and let
\[B = \{x_1, \dots , x_n\}\]
be a basis of L. Then the products 
\[x_{i_1} \otimes \dots x_{i_m}\]
\[1\leq i_j \leq n \]
form a basis of $T_m(L)$. We can define a total order on elements $x_{i(j)}$ by setting
\[x_{i_n} < x_{i_m} \; \text{ if } \;  i_n \leq i_m\] 
A nondecreasing sequence of elements of B is
\[(x_{i_1}, \dots, x_{i_m})\]
\[1 \leq i_1 \leq \dots \leq i_m \leq n\]
\end{definition}

\begin{definition}
For an element
\[u = x_{i_1} \otimes \dots x_{i_n} \in T_n(L)\]
define
\[D(u) = \#\{(x_{i_k},x_{i_m}) | x_{i_k} > x_{i_m} \text{ and } k < m\}\]
where \# we denote the number of elements of this set. For an element of the tensor algebra
\[u = \sum_i c_i u_i \in T(L)\]
\[u_i \in T_i(L) \; , \; c_i \in K\]
define
\[D(u) = sup\{D(u_i)\}\]
The number $D(u)$ is called the disorder of $u$.
\end{definition}
Suppose we are given a basis $(E_i)_{i \in I}$ of L such that all elements $E_i$ are homogeneous and such that the index set $I$ is totally ordered. For every integer $r \geq 0$ let $H_r$ be the set of all finite sequences $(i_1,\dots,i_r)$ in $I$ such that 
\[r \geq 0 \text{ arbitrary}\]
\[i_1 \leq i_2 \leq \dots \leq i_r\]
\[i_p < i_{p+1} \text{ if } E_{i_p} \text{ and } E_{i_{p+1}} \text{ are odd.}\]
define
\[H_0 = \emptyset\]
\[H = \cup_{r \geq 0} H_r\]
The proof of the following lemmas can be found in \cite{scheunert}.
\begin{lemma}
Define for every $i \in I$
\[F_i = \begin{cases} E_i\otimes 1, & \mbox{if } E_i  \in L_0 \\ 1\otimes E_i, & \mbox{if } E_i  \in L_1 \end{cases}\]
and, furthermore, for every $r \geq 0$ and every element $N = (i_1, \dots, i_r) \in H_r$
\[F_N = F_{i_1} \dots F_{i_r}\]
(by convention, $F_{\emptyset} =1$). If $N$ runs through all sequences from $H_r$ the elements $F_N$ form a basis of $S_r(L)$.
\end{lemma}
\begin{lemma}\label{keylemmapwb}
There exists a graded representation $g$ of the Lie superalgebra L in the vector space S(L) such that for every $i \in I$
\[g(E_i)F_N = F_iF_N \; \; \; \text{ if } N \in H \text{ and } i \leq N\]
\end{lemma}

\begin{theorem}[Poincar\'e-Birkhoff-Witt]
Let $L$ be a Lie superalgebra with a basis $B$. The monomials formed by finite nondecreasing sequences of elements in $B$ constitute basis of the universal enveloping algebra $U(L)$.
\end{theorem}

\begin{proof}
Let $P: T(L) \to U(L)$ be the canonical morphism, M the submodule generated be the monomials described in the formulation of the theorem. We thave to prove that $U(L) = M$. Note that
\[ U(L) = \sum_p P(T_p)\]
If p=1, then $T_1 \subset L$ ; it follow that $P(T_1) \subset M$. Suppose that $P(T_r) \subset M$. It suffices to prove that $P(T_{r+1}) \subset M$.

Define $T_r^u$ as the submodule of $T_r$ generated by elements with disorder $\leq$ u, and proceed by a second induction on the disorder. We have $P(T_{r+1}^0) \subset M$. Suppose that $v=a\otimes x \otimes y \otimes b \in T^u_{r+1}$, where $x > y \in B$, and $a \in T^n, b \in T_m$ monomials formed by the basic elements in B. Then
\[P(v) = P(a\otimes (-1)^{xy} y\otimes x \otimes b) + P(a \otimes [x,y] \otimes b)\]
but
\[P(a\otimes (-1)^{xy}y\otimes x \otimes b) \in T_{r+1}^{u-1} \subset M\]
and
\[P(a \otimes [x,y] \otimes b) \subset M \]
by first induction hypothesis. Hence, $P(u) \in M$ and it follows that $P(T_{r+1}) \subset M$.

It remains to prove linear independence. For a given sequence $\sigma = (x_{o_1}, \dots , x_{o_n})$ of nondecreasing elements of B; define 
\[x_o = x_{o_1} \dots  x_{o_n} \in U(L).\] 
Suppose
\[\sum_i c_i x_{w_i} = 0\]
where $w_i$ is a sequence nondecreasing, and $c_i \in K$. Using \ref{keylemmapwb}
\[\sum_i c_i x_{w_i} \cdot 1 = \sum_i c_i F_{w_i},\]
and because of the linear independence of the $ F_{w_i}$, it follows that $c_i = 0 \; \forall i$.
\end{proof}
\begin{remark}
A lemma similar to lemma \ref{keylemmapwb} for the case of a Klein graded lie algebra can be found in \cite{PBW} and one can prove the previous theorem for the Klein graded case using this lemma for the step of the linear independence.
\end{remark}

\subsection{Hopf structure}
\subsubsection{Hopf algrebra}
A Hopf algebra, is a structure that is simultaneously an (unital associative) algebra and a (counital coassociative) coalgebra, with these structures' compatibility making it a bialgebra, and that moreover is equipped with an antiautomorphism satisfying a certain property. The representation theory of a Hopf algebra is particularly nice, since the existence of compatible comultiplication, counit, and antipode allows for the construction of tensor products of representations, trivial representations, and dual representations. Here we will give only the definition of a Hopf algebra in order to identify the hopf structure in the enveloping algebra of some graded Lie algebras.
\begin{definition}
An associative algebra A (with product p and unit 1) is a bialgebra over the ground field $K$ if there are maps
\begin{align*}
c&: A \to A \otimes A \text{ (coproduct)  }
\\ \epsilon& : A \to K \text{ (counit) }
\\ S& : A \to A \text{ (andipode)}
\end{align*}
such that
\begin{align*}
(Id_A \otimes c) \circ c = (c \otimes Id_A)\circ c
\\(Id_A \otimes \epsilon) \circ c = Id_A = (\epsilon \otimes Id_A) \circ c
\end{align*}
and the following diagrams commute

\[
\begin{tikzcd}
&A\otimes A \arrow{d}{c\otimes c}\arrow{r}{p}
&A \arrow{r}{c} 
&A\otimes A \arrow{d}{p \otimes p}\\
& A\otimes A\otimes A\otimes A\arrow{rr}{Id\otimes t \otimes Id}
&{}
& A\otimes A\otimes A\otimes A\arrow
\end{tikzcd}
\]
where $t:A\otimes A \to A \otimes A$ is the linear map defined by 
\[t(x \otimes y ) = y \otimes x \; \; \; \text{,  for all x,y in A. }\]

\[
\begin{tikzcd}
&A\otimes A \arrow{d}{\epsilon \otimes \epsilon} \arrow{r}{p}
&A \arrow{ld}{\epsilon} \\
&K\otimes K
\end{tikzcd}
\begin{tikzcd}
&K\otimes K \arrow{d}{1 \otimes 1} \arrow{rd}{1}\\
&A\otimes A  
&A \arrow{l}{c} 
\end{tikzcd}
\begin{tikzcd}
&K \arrow{d}{Id} \arrow{r}{1}
&A \arrow{ld}{\epsilon} \\
&K
\end{tikzcd}
\]
\end{definition}
Now we can give the definition for a Hopf algebra.
\begin{definition}
A Hopf algebra is a bialgebra $A$ over a field $K$ together with a linear map $S: A \to A$ (called the antipode) such that the following diagram commutes:
\[
\begin{tikzcd}
&A \otimes A \arrow{r}{S \otimes Id}
&A \otimes A \arrow{rd}{p} \\
&A \arrow{d}{c} \arrow{u}{c} \arrow{r}{\epsilon}
&K \arrow{r}{1}
&A \\
&A \otimes A \arrow{r}{Id \otimes S}
&A \otimes A \arrow{ur}{p}
\end{tikzcd}
\]
\end{definition}
\subsubsection{The enveloping algebra of a Lie algebra as a Hopf algebra}
Given the definition of a Hopf algebra, we proceed by identifying the Hopf structure on the enveloping algebra of a Lie algebra. We will see that the enveloping algebra of a Lie algebra become a Hopf algebra if we define properly the coproduct, the counit and the antipode maps. Now, given a Lie algebra $L$, it is known that the enveloping algebra $U(L)$ is an associative algebra. We define the comultiplication as
\[c(x) = x \otimes 1 + 1 \otimes x\]
for every x in L. It is compatible with the commutator and can therefore be uniquely extended to all U(L).
Next we define the counit as
\[ \epsilon(x) = 0 \]
for all x in L. And finally, we define the antipode as
\[S(x) = -x\]
\subsubsection{The enveloping algebra of a Lie superalgebra as a Hopf superalgebra}
Initialy we have to take a look on the graded tensor product. Let $A$ and $B$ be two associative superalgebras, then the vector space $A\otimes B$ has a $Z_2$ gradation too
\[(A \otimes B)_c = \bigoplus_{a+b=c} (A_a \otimes B_b) \; \; , \; \;c \in Z_2\]
Now, on $A \otimes B$ we define a multiplication by the requirement that
\[(x \otimes y)(x' \otimes y') = (-1)^{ba'}(xx') \otimes (yy')\]
\[\text{for all } \; x \in A, y\in B_b, x' \in A_{a'} ,y' \in B \; \; ; \; \; b,a' \in Z_2\] 
It's easy to see that with this multiplication $A\otimes B$ is an associative superalgebra. It will be called the graded tensor product and we denote this algebra by $A \bar{\otimes} B$.

Let L be a Lie superalgebra, and let $\sigma:L \to U(L)$ be the canonical mapping. Evidently, the diagonal mapping
\begin{equation*}
L \to L \times L \hspace{10pt} , \hspace{10pt} A \to (A,A) \text{ if A } \in L
\end{equation*}
is a homomorphism of Lie superalgebras. Hence, due to the proposition \ref{maponassociative} there exists a unique homomorphism of superalgebras
\begin{equation*}
c : U(L) \to U(L) \bar{\otimes} U(L)
\end{equation*}
such that
\begin{equation*}
c(\sigma(A)) =\sigma(A)\otimes 1 + 1 \otimes \sigma(A) \hspace{10pt} \text{for all A } \in L
\end{equation*}
\begin{equation*}
c(1) = 1 \otimes 1 .
\end{equation*}
the homomorphism c is called the coproduct of the enveloping algebra U(L).
\begin{remark}
c is associative. This is to say that
\begin{equation*}
(c\otimes Id_U) \circ c = (Id_U \otimes c ) \circ c,
\end{equation*}
both sides being algebra-homomorphisms of U(L) into $U(L) \bar{\otimes} U(L) \bar{\otimes} U(L)$. (Recall that $U(L) \bar{\otimes} U(L) \bar{\otimes} U(L)$ ,  $ (U(L) \bar{\otimes} U(L)) \bar{\otimes} U(L)$ and $U(L) \bar{\otimes} ( U(L) \bar{\otimes} U(L))$ are canonically isomorphic.)
\end{remark}
\begin{proof}
\begin{align*}
(c\otimes Id_U) \circ c (\sigma(A)) &= (c\otimes Id_U) (\sigma(A) \otimes 1 + 1 \otimes \sigma(A)) 
\\&=(\sigma(A) \otimes 1 + 1 \otimes \sigma(A)) \otimes 1 + (1\otimes1) \otimes\sigma(A)
\\&=\sigma(A) \otimes 1 \otimes 1 + 1 \otimes \sigma(A) \otimes 1 + 1\otimes1 \otimes\sigma(A)
\\&= 1 \otimes (\sigma(A) \otimes 1 + 1 \otimes\sigma(A)) + \sigma(A) \otimes 1 \otimes 1
\\&=(Id\otimes c) \circ c (\sigma(A))
\end{align*}
\end{proof}

\begin{remark}
Consider K (the ground field) as an associative superalgebra (the odd subspace of K being equal to \{0\}).  There exists a unique homomorphism of superalgebras
\begin{equation*}
\epsilon : U(L) \to K
\end{equation*}
such that
\begin{equation*}
\epsilon \circ \sigma = 0 \hspace{10pt},\hspace{10pt} \epsilon(1) =1
\end{equation*}
Identifying U(L) $\otimes$ K and K $\otimes$ U(L) with U(L) canonically we have
\begin{equation*}
(\epsilon \otimes id_U) \circ c = (id_U \otimes \epsilon) \circ c = id_U.
\end{equation*}
the homomorphism $\epsilon$ is called the counit of U(L).
\end{remark}
\begin{remark}
Let
\[\mu : U(L) \bar{\otimes} U(L) \to U(L)\]
be the linear mapping defined by the multiplication in $U(L)$ and let
\[\bar{\epsilon} : U(L) \to U(L) \]
be defined by
\[\bar{\epsilon}(X) = \epsilon(X) \cdot 1_U \; \;\text{ for all } x \in U(L) .\]
Also, define the linear mapping
\[S : U(L) \to U(L)\]
such that
\[S(XY) = (-1)^{xy}S(X)S(Y)\]
\[\text{for all } X \in U(L)_a \; , \; Y \in U(L)_b \; \; ; a,b \in Z_2\]
\[S(S(A)) = - S(A)\]
\[\text{for all } \; A \in L\]
\[S(1) = 1\]
and from these properties we deduce that
\[S^2 = id.\]
Then
\[\mu \circ(S \otimes id_U)\circ c = \mu \circ (id_U \otimes S)\circ c = \bar{\epsilon}.\]
A linear mapping $\delta$ with this property is called an antipode.
\end{remark}
In view of these remarks the superalgebra U(L), equipped with the coproduct c, is what is called a Hopf superalgebra.

\subsubsection{The enveloping algebra of a Klein graded Lie algebra as a Klein graded Hopf algebra}
Having the work done for a Lie superalgebra makes it easy to see what happen to a Klein graded Lie algebra by using the same steps with respect to the gradation. Let $A$ and $B$ be two  Klein graded associative algebras, then the vector space $A\otimes B$ has a $Z_2$ gradation too
\[(A \otimes B)_c = \bigoplus_{ab=c} (A_a \otimes B_b) \; \; , \; \;c \in \mathcal{K}\]
Now, on $A \otimes B$ we define a multiplication by the requirement that
\[(x \otimes y)(x' \otimes y') = \theta(b,a')(xx') \otimes (yy')\]
\[\text{for all } \; x \in A, y\in B_b, x' \in A_{a'} ,y' \in B \; \; ; \; \; b,a' \in \mathcal{K}\] 
It's easy to see that with this multiplication $A\otimes B$ is an  Klein graded associative algebra. It will be called the graded tensor product and we denote this algebra by $A \bar{\otimes} B$.

Let $L=L_e \oplus L_r \oplus L_s \oplus L_t$ be a  Klein graded Lie algebra, and let $\sigma:L \to U(L)$ be the canonical mapping. Evidently, the diagonal mapping
\begin{equation*}
L \to L \times L \hspace{10pt} , \hspace{10pt} A \to (A,A) \text{ if A } \in L
\end{equation*}
is a homomorphism of  Klein graded Lie algebras. Hence, due to the proposition \ref{maponassociative} there exists a unique homomorphism of  Klein graded algebras
\begin{equation*}
c : U(L) \to U(L) \bar{\otimes} U(L)
\end{equation*}
such that
\begin{equation*}
c(\sigma(A)) =\sigma(A)\otimes 1 + 1 \otimes \sigma(A) \hspace{10pt} \text{for all A } \in L
\end{equation*}
\begin{equation*}
c(1) = 1 \otimes 1 .
\end{equation*}
the homomorphism c is called the coproduct of the enveloping algebra U(L).
\begin{remark}
c is associative. This is to say that
\begin{equation*}
(c\otimes Id_U) \circ c = (Id_U \otimes c ) \circ c,
\end{equation*}
both sides being algebra-homomorphisms of U(L) into $U(L) \bar{\otimes} U(L) \bar{\otimes} U(L)$. (Recall that $U(L) \bar{\otimes} U(L) \bar{\otimes} U(L)$ ,  $ (U(L) \bar{\otimes} U(L)) \bar{\otimes} U(L)$ and $U(L) \bar{\otimes} ( U(L) \bar{\otimes} U(L))$ are canonically isomorphic.)
\end{remark}
\begin{proof}
\begin{align*}
(c\otimes Id_U) \circ c (\sigma(A)) &= (c\otimes Id_U) (\sigma(A) \otimes 1 + 1 \otimes \sigma(A)) 
\\&=(\sigma(A) \otimes 1 + 1 \otimes \sigma(A)) \otimes 1 + (1\otimes1) \otimes\sigma(A)
\\&=\sigma(A) \otimes 1 \otimes 1 + 1 \otimes \sigma(A) \otimes 1 + 1\otimes1 \otimes\sigma(A)
\\&= 1 \otimes (\sigma(A) \otimes 1 + 1 \otimes\sigma(A)) + \sigma(A) \otimes 1 \otimes 1
\\&=(Id\otimes c) \circ c (\sigma(A))
\end{align*}
\end{proof}

\begin{remark}
Consider K (the ground field) as a  Klein graded associative algebra (the subspaces $L_r$,$L_s$ and $L_t$ of K being equal to \{0\}).  There exists a unique homomorphism of  Klein graded algebras
\begin{equation*}
\epsilon : U(L) \to K
\end{equation*}
such that
\begin{equation*}
\epsilon \circ \sigma = 0 \hspace{10pt},\hspace{10pt} \epsilon(1) =1
\end{equation*}
Identifying U(L) $\otimes$ K and K $\otimes$ U(L) with U(L) canonically we have
\begin{equation*}
(\epsilon \otimes id_U) \circ c = (id_U \otimes \epsilon) \circ c = id_U.
\end{equation*}
the homomorphism $\epsilon$ is called the counit of U(L).
\end{remark}
\begin{remark}
Let
\[\mu : U(L) \bar{\otimes} U(L) \to U(L)\]
be the linear mapping defined by the multiplication in $U(L)$ and let
\[\bar{\epsilon} : U(L) \to U(L) \]
be defined by
\[\bar{\epsilon}(X) = \epsilon(X) \cdot 1_U \; \;\text{ for all } x \in U(L) .\]
Also, define the linear mapping
\[S : U(L) \to U(L)\]
such that
\[S(XY) = \theta(x,y)S(X)S(Y)\]
\[\text{for all } X \in U(L)_a \; , \; Y \in U(L)_b \; \; ; a,b \in \mathcal{K}\]
\[S(S(A)) = - S(A)\]
\[\text{for all } \; A \in L\]
\[S(1) = 1\]
and from these properties we deduce that
\[S^2 = id.\]
Then
\[\mu \circ(S \otimes id_U)\circ c = \mu \circ (id_U \otimes S)\circ c = \bar{\epsilon}.\]
A linear mapping $\delta$ with this property is called an antipode.
\end{remark}
In view of these remarks the  Klein graded algebra U(L), equipped with the coproduct c, is what is called a  Klein graded Hopf algebra.

\section{Representations}\label{representations}
\subsection{The connection between representation of L and U(L)}

\begin{definition}
Let L be a Lie superalgebra and let V be a $Z_2$-graded vector space. Recall that $End(V)$ has a natural $Z_2$-gradation which converts it into an assiciative superalgebra.
A graded representation r of L in V is an even linear mapping
\begin{equation}
r : L \to Hom(V)
\end{equation}
such that
\begin{equation}\label{repdef}
r([A,B]) = r(A)r(B) - (-1)^{ab}r(B)r(A)
\end{equation}
\begin{center}
for all $A \in L_a, B \in L_b ; a,b \in Z_2$
\end{center}
\end{definition}
A $Z_2$-graded vector space equipped with a graded representation of L is called a (left) graded L-module.

\begin{remark}
The definition makes sense even if the vector space V is not graded provided we drop the requirement that r should be even. In this work we shall have no occasion to discuss these "non-graded" representations.
\end{remark}
Let U(L) be the enveloping algebra of L and let $s : L \to U(L)$ be the canonical mapping. From now on we shall identify L with a graded subspace of U(L) by means of s. Under this identification s is just the injection of L into U(L).

Let r be a graded representation of L in some graded vector space V. Due to the universal property of U(L) there exists a unique homomorphism of associative superalgebras
\begin{equation}
\bar{r} : U(L) \to End(V)
\end{equation}
which extends r, i.e. such that
\begin{equation}
\bar{r}(A) = r(A) \text{ for all } A \in L \hspace{10pt},\hspace{10pt} \bar{r}(1) = id.
\end{equation}
\[
\begin{tikzcd}
&L \arrow{d}{Id_L} \arrow{r}{r}
&EndV  \\
&U(L)\arrow{ur}{\bar{r}}
\end{tikzcd}
\]
In particular, we have
\begin{equation}
\bar{r}(U(L)_a)V_b \subset V_{a+b} \text{ for all } a,b \in Z_2.
\end{equation}
Hence $\bar{r}$ is a graded representation of the associative superalgebra U(L) in the graded vector space V or, using the modules language, V is a $Z_2$-graded left U(L)-module.
Conversely, suppose we are given a $Z_2$-graded left U(L)-module V; let
\begin{equation}
\omega : U(L) \to End(V)
\end{equation}
be the corresponding homomorphism of associative superalgebras. Then the restriction r of $\omega$ to L is a graded representation of L in V and $\bar{r} = \omega$.

Respectivly, we should give the definition of graded representation of a  Klein graded Lie algebra on a Klein graded vector space.

\begin{definition}
Let L be a  Klein graded Lie algebra and let V be a Klein graded vector space. Recall that $End(V)$ has a natural Klein gradation which converts it into an assiciative Klein graded algebra.
A graded representation r of L in V is an  linear mapping of degree e
\begin{equation}
r : L \to End(V)
\end{equation}
such that
\begin{equation}\label{repdef2}
r([A,B]) = r(A)r(B) - \theta(a,b)(B)r(A)
\end{equation}
\begin{center}
for all $A \in L_a, B \in L_b ; a,b \in \mathcal{K}$
\end{center}
\end{definition}

In view of this discussion the concepts of a "graded representation of L", a "graded L-module" and a "left graded U(L)-module" are completely equivalent for both Lie superalgebras and Klein graded Lie algebras. It depends on the circumstances which language is preffered.
\begin{definition}
A simple graded module(or equivalently a simple graded representation) is a module that does not contains any non-trivial submodules.
\end{definition}
\begin{proposition}
Let $V = V_0 \oplus V_1$ be a $Z_2$-graded vector space and let $V' = V_0' \oplus V_1'$ be the $Z_2$-graded vector space whose underlying vector space is equal to that of V but whose gradation is definded by
\begin{equation*}
V_0' = V_1 \hspace{10pt} ,  \hspace{10pt} V_1' = V_0.
\end{equation*}
Then $End(V) = End(V')$. 
\end{proposition}
\begin{proof}
Let f $\in End(V)$ of degree $d \in Z_2$ such that
\begin{equation*}
f(V_a) \subset V_{a+d} \hspace{10pt}, \forall a \in Z_2.
\end{equation*}
or
\begin{equation*}
f(V_{a+1}) \subset V_{(a+1)+d} \hspace{10pt}, \forall (a+1) \in Z_2.
\end{equation*}
and
\begin{equation*}
f(V'_{a}) \subset V_{a+d}' \hspace{10pt}, \forall (a+1) \in Z_2.
\end{equation*}
and thus, f $\in End(V')$ of degree  $d \in Z_2$. We can repeat the same proccess backwards and the proof is complete.
\end{proof}

\begin{proposition}\label{76prop}
Let $V = V_e \oplus V_r \oplus V_s \oplus V_t$ be a Klein-graded vector space and let $V' = V_e' \oplus V_r' \oplus V_s' \oplus V_t'$ be the Klein-graded vector space whose underlying vector space is equal to that of V but whose gradation is defined by
\begin{equation*}
V_k' = V_{\sigma (k)} \hspace{10pt}, \forall k \in \mathcal{K}
\end{equation*}
and $\sigma$ is a permutation of the Klein group that satisfies the condition
\begin{equation*}
\sigma (ab) = \sigma(a) \sigma(b) \hspace{10pt}, \forall a,b \in \mathcal{K}.
\end{equation*}
Then $End(V) = End(V')$. Furthermore, a homomorphism of degree d in $End(V)$ is a homomorphism of degree $\sigma^{-1} (d)$ in $End(V')$.
\end{proposition}
\begin{proof}
Let f $\in$ Hom(V) of degree $d \in \mathcal{K}$ such that
\begin{equation*}
f(V_a) \subset V_{ad} \hspace{10pt}, \forall a \in \mathcal{K}.
\end{equation*}
Then there is a k such that $\sigma(k) = d$ and 
\begin{equation*}
f(V_a')  = f(V_{\sigma (a)}) \subset V_{\sigma (a) d} = V_{\sigma (a) \sigma(k)} =V_{\sigma (ak)} = V_{ak}' \hspace{10pt}, \forall a \in \mathcal{K}.
\end{equation*}
So, f is in Hom(V') and it's degree is k = $\sigma^{-1} (d)$.
\end{proof}
Consequently,on both $Z_2$-graded and Klein graded case, a graded representation of L in V is also a graded representation of L in $V'$ and vice versa. Nevertheless, these two representations are not necessarily isomorphic.

\subsection{Invariants}
\begin{definition}
Let L be a graded Lie algebra and let V be a graded L-module. An element x of V is called invariant with respect to the given representantion of L in V (or simply L-invariant) if
\begin{equation*}
Ax = 0 \hspace{10pt} \text{ for all } A \in L
\end{equation*}
\end{definition}
The set of all L-invariant elements of V is denoted by $V^L$. An element of V is L-invariant if and only if its homogeneous components are L-invariant. Hence $V^L$ is a graded subspace of V.

\begin{example}
Let L be a Lie superalgebra and let V be a graded L-module. If x is a homogeneous element of V then
\begin{equation*}
C(x) = \{A \in L | Ax = 0\}
\end{equation*}
is a graded subalgebra of L and x is C(x)-invariant. We call C(x) the graded subalgebra of L consisting of those elements of L which leave x invariant.
\end{example}

\subsubsection{Schur's lemma for Lie superalgebras}

\begin{lemma}
Suppose that the field K is algebraically closed. Let L be a Lie superalgebra and let V be a finite-dimensional simple graded L-module. Then
\begin{equation*}
Hom_L(V)_0 = K Id \hspace{10pt},\hspace{10pt} Hom_L(V)_1 = K u,
\end{equation*}
where either u=0 or else $u^2$=Id.
\end{lemma}
Before we write the proof, we have to remember the schur lemma for an associative algebra with unit element which states that:
\begin{lemma}
Let V be a finite dimensional irreducible representation of an associative algebra A with unit element over an algebraically closed field k, and $f : V \to  V$ is a homorphism such that
\[
f(a)f(b)v = f(b)f(a)v
\]
\[\text{for all } a,b \in A, v \in V.\]
Then $f = c Id$ for some c in K.
\end{lemma}
Details and a proof on this lemma can be found on \cite{manyauthors}. Now we can prove the shcur's lemma for a Lie superalgebra case:
\begin{proof}
First we have to prove that both homogeneous compoments $V_0$ and $V_1$ are simple $U(L)_0$-modules. Suppose that $V_0$ isn't simple $U(L)_0$-module, then we can find a subspace $W_0$ of $V_0$ such that
\begin{equation*}
U(L)_0 W_0 \subset W_0
\end{equation*}
now set 
\begin{equation*}
W = U(L) W_0
\end{equation*}
it is not equal to zero due to the selection and we get that
\begin{equation}
\begin{split}
  L_0 W \subset L_0 U(L) W_0 \subset U(L) W_0 =W\\
  L_1 W \subset L_1 U(L) W_0 \subset U(L) W_0 =W\\
\end{split}
\end{equation}

Hence, we have that W is a L-submodule of V which is a contradiction. The same holds for $V_1$ $U(L)_0$-module, so both of them are simple $U(L)_0$-modules. 

Now, let f in $Hom(V)_0$. Since $U(L)_0$ is an unitary assosiative algebra, the restriction of f on $V_0$(resp. on $V_1$) is some constant $c_0$(resp. $c_1$). Since V is simple, schur's lemma for associative algebras appLies again and we have $c_0 = c_1$ and so $f = c_0 id$

Now, let f in $Hom(V)_1$. Then$ f^2$ is in $Hom(V)_0$ and so $f^2 = c id$ for some $c \in K$. If c=0 then $Kerf \neq 0$ (remember that kernel is a L-submodule of V) and so f=0. Otherwise, set $g = f/\sqrt{c}$. Let w be another odd endomorphism such that $w^2 = id$. Then (w + g), (w-g) are also odd endomorphisms. Therefore (w-g) (resp. (w+g)) is either isomorphism or zero. Since  (w-g)(w+g) = 0, it impLies that w = g or w=-g, on any case $f = c' id$.
\end{proof}

\subsubsection{Schur's lemma for Klein-graded Lie algebras}
\begin{lemma}\label{71lemma}
Suppose that the field K is algebraically closed. Let L be a Klein-graded Lie algebra and let V be a finite-dimensional simple graded L-module. Then
\begin{equation*}
Hom_L(V)_e = K Id \hspace{10pt},\hspace{10pt} Hom_L(V)_r = K u_r,\hspace{10pt} Hom_L(V)_s = K u_s,\hspace{10pt} Hom_L(V)_t = K u_t
\end{equation*}
where either $u_k=0$ or else $u_k^2$=Id for all k $\in \mathcal{K}  - \{e\}$
\end{lemma}
\begin{proof}
First we have to prove that all homogeneous compoments $V_e$ , $V_r$,$V_s$,$V_t$ are simple $U(L)_e$-modules. Suppose that $V_e$ isn't simple $U(L)_e$-module, then we can find a subspace $W_e$ of $V_e$ such that
\begin{equation*}
U(L)_e W_e \subset W_e
\end{equation*}
now set 
\begin{equation*}
W = U(L) W_e
\end{equation*}
it is not equal to zero due to the selection and we get that
\begin{equation}
\begin{split}
  L_e W \subset L_e U(L) W_e \subset U(L) W_e =W\\
  L_r W \subset L_r U(L) W_e \subset U(L) W_e =W\\
  L_s W \subset L_s U(L) W_e \subset U(L) W_e =W\\
  L_t W \subset L_t U(L) W_e \subset U(L) W_e =W\\
\end{split}
\end{equation}

Hence, we have that W is a L-submodule of V which is a contradiction. The same holds for $V_r$,$V_s$,$V_t$ $U(L)_e$-module, so all of them are simple $U(L)_e$-modules. 

Now, let f in $Hom(V)_e$. Since $U(L)_e$ is an unitary assosiative algebra, the restriction of f on $V_e$(resp. on $V_r$) is some constant $c_e$(resp. $c_r$). Since V is simple, schur's lemma for associative algebras appLies again and we have $c_e = c_r$ and so $f = c_e id$

Now, let f in $Hom(V)_k$ for some k in \{r,s,t\}. Then$ f^2$ is in $Hom(V)_e$ and so $f^2 = c id$ for some $c \in K$. If c=0 then $Kerf \neq 0$ (remember that kernel is a L-submodule of V) and so f=0. Otherwise, set $g = f/\sqrt{c}$. Let w be another odd endomorphism such that $w^2 = id$. Then (w + g), (w-g) are also odd endomorphisms. Therefore (w-g) (resp. (w+g)) is either isomorphism or zero. Since  (w-g)(w+g) = 0, it implies that w = g or w=-g, on any case $f = c' id$.
\end{proof}

\begin{remark}\label{712remark}
We know that if a homomorphism f is in $Hom(V)_a$ and g is in $Hom(V)_b$ then fg is in $Hom(V)_ab$. This applied to the previous homomorphisms lead us to:
\begin{equation*}
u_a u_b = u_{ab}
\end{equation*}
for a,b in \{r,s,t\}.
\end{remark}
\subsubsection{Lie superalgebra: The supertrace and the killingform}
Suppose we are given a Lie superalgebra L and three graded L-modules V,W and U. A billinear mapping $g:V \times W \to U$ which is homogeneous of degree b is L-invariant if and only if
\begin{equation*}
Ag(x,y) = (-1)^{ab}g(Ax,y) +  (-1)^{a(b+c)}g(x,Ay)
\end{equation*}
\begin{center}
for all $A \in L_a, x \in V_c, y \in W ; a,c \in Z_2$
\end{center}
For every bilinear mapping  $g:V \times W \to U$ we define a bilinear mapping  $sg:V \times W \to U$ by
\begin{equation*}
sg(y,x) = (-1)^{cr}g(x,y)
\end{equation*}
\begin{center}
for all $x \in V_c, y \in W_r ; c,r \in Z_2$.
\end{center}
The mapping $g \to sg$ of B(V,W;U) into B(W,V;U) is an isomorphim of graded L-modules, hence g is L-invariant if and only if sg is L-invariant. This leads to the following definition.

\begin{definition}
Let U and V be two $Z_2$-graded vector spaces. A bilinear mapping of $V \times V$ into U is called supersymmetric / skew-supersymmetric if $sg \pm g$, i.e. if
\begin{equation*}
b(y,x) = \pm (-1)^{ab}b(x,y)
\end{equation*}
\begin{center}
for all $x \in V_a, y \in V_b ; a,b \in Z_2$.
\end{center}
\end{definition}

\begin{example}
Recall that the product mapping in a Lie superalgebra is skew-supersymmetric.
\end{example}

Let V be a finite-dimensional $Z_2$-graded vector space and let
\begin{equation*}
\gamma : V \to V
\end{equation*}
be the linear mapping which satisfies
\begin{equation*}
\gamma(x) = (-1)^a x \hspace{10pt} \text{ if } x \in V_a ; a \in Z_2.
\end{equation*}
We define a linear form str on the general linear Lie superalgebra pl(V) by
\begin{equation*}
str(A) = Tr(\gamma A) \hspace{10pt} \text{ for all } A \in pl(V)
\end{equation*}
and call str the supertrace. The linear form str is even and pl(V)-invariant:
\begin{equation*}
str([A,B]) = 0  \hspace{10pt} \text{ for all } A,B \in pl(V)
\end{equation*}
or, equivalently,
\begin{equation*}
str(AB) = (-1)^{ab}str(BA)
\end{equation*}
\begin{center}
for all $A \in  Hom(V)_a$ , $B \in Hom(V)_b$ ; $a,b \in Z_2$.
\end{center}

\begin{definition}
Let L be a Lie superalgebra and let V be a finite-dimensional graded L-module. The n-linear form on L which is defined by 
\begin{equation*}
(A^1,.......,A^n) \to str(A^1_V.......A^n_V)
\end{equation*}
\begin{center}
for all $A^i \in L$ , $1 \leq i \leq n$,
\end{center}
is even and L-invariant (with respect to the adjoint representation of L).
\end{definition}
The n-linear form defined above called the n-linear form associated with the graded L-module V. The most important case is obtained if we choose V=L and n=2. This special case is described below, and it is known as killing form.
\begin{definition}
Let L be a finite-dimensional Lie superalgebra and let $\gamma$ be the automorphism of L such that
\begin{equation*}
\gamma(A) = (-1)^a A \text{  for all } A \in L_a ; a \in Z_2.
\end{equation*}
The bilinear form $\phi$ which is defined by
\begin{equation*}
\phi(A,B) = str(adA adB) = Tr(\gamma adA adB)
\end{equation*}
for all A,B $\in L$. $\phi$ is called the killing form of L ; it is even, invariant and supersymmetric.
\end{definition}

\begin{proposition}
Let L be a Lie superalgebra and let H be a finite-dimensional graded ideal of L. Supose we are given a non-degenerate homogeneous bilinear form b on H of degree $\beta$ and a homogeneous n-linear form h of H of degree $\zeta$. We assume that b and h are L-invariant if the ideal H considered as a graded submodule of  L.

Let $(E_i)_{1 \leq i \leq p}$ be a basis of H consisting of homogeneous elements and let $\epsilon_i$ be the degree of $E_i$. We introduce a second basis $(F_i)_{1 \leq i \leq p}$ of H by the condition that
\begin{equation*}
b(F_j,E_i) = \delta_{ij} \hspace{10pt} , \hspace{10pt} 1 \leq i,j \leq p.
\end{equation*}
Define
\begin{equation*}
X = \sum_{1 \leq i_1,...,i_n \leq p} (-1)^{\beta \sigma(i_1,...,i_n)} h(E_{i_1},...,E_{i_n}) F_{i_n}....F_{i_1}
\end{equation*}
\begin{center}
with $\sigma(i_1,...,i_n) = \sum^n_{s=1}s\epsilon_{i_s}$.
\end{center}
Then X  is a homogeneous L-invariant element of U(L) of degree $\zeta + n\beta$, hence
\begin{equation*}
XY = (-1)^{(\zeta + n\beta)}YX \hspace{10pt} \text{for all } Y \in U(L)_a ; a \in Z_2.
\end{equation*}
The element X does not depend on the choice of the basis $(E_i)_{1 \leq i \leq p}$.
\end{proposition}
A proof can be found on \cite{scheunert}. The element X is called generalized Casimir element of U(L), it supercommutes with every element of U(L).

\subsubsection{Klein-graded Lie algebra: The supertrace and the killingform}
Let L be a Klein-graded Lie algebra and three graded L-modules V,W and U. A billinear mapping $g:V \times W \to U$ which is homogeneous of degree b is L-invariant if and only if
\begin{equation*}
Ag(x,y) = \theta(a,b)g(Ax,y) + \theta(a,bc)b(x,Ay)
\end{equation*}
\begin{center}
for all $A \in L_a, x \in V_c, y \in W ; a,c \in\mathcal{K}$
\end{center}
For every bilinear mapping  $g:V \times W \to U$ we define a bilinear mapping  $sg:V \times W \to U$ by
\begin{equation*}
sg(y,x) = \theta(c,r)g(x,y)
\end{equation*}
\begin{center}
for all $x \in V_c, y \in W_r ; c,r \in \mathcal{K}$.
\end{center}
The mapping $g \to sg$ of B(V,W;U) into B(W,V;U) is an isomorphim of graded L-modules, hence g is L-invariant if and only if sg is L-invariant. This leads to the following definition.

\begin{definition}
Let U and V be two Klein graded vector spaces. A bilinear mapping of $V \times V$ into U is called supersymmetric / skew-supersymmetric if $sg \pm g$, i.e. if
\begin{equation*}
b(y,x) = \pm \theta(a,b)b(x,y)
\end{equation*}
\begin{center}
for all $x \in V_a, y \in V_b ; a,b \in\mathcal{K}$.
\end{center}
\end{definition}

\begin{example}
Recall that the product mapping in a Lie superalgebra is skew-supersymmetric.
\end{example}

Let V be a finite-dimensional Klein graded vector space and let
\begin{equation*}
\gamma : V \to V
\end{equation*}
be the linear mapping which satisfies
\begin{equation*}
\gamma(x) = \theta(a,a)x \hspace{10pt} \text{ if } x \in V_a ; a \in \mathcal{K}.
\end{equation*}
We define a linear form str on the general linear Lie superalgebra pl(V) by
\begin{equation*}
str(A) = Tr(\gamma A) \hspace{10pt} \text{ for all } A \in pl(V)
\end{equation*}
and call str the supertrace. The linear form str is even(degree e) and pl(V)-invariant:
\begin{equation*}
str([A,B]) = 0  \hspace{10pt} \text{ for all } A,B \in pl(V)
\end{equation*}
or, equivalently,
\begin{equation*}
str(AB) = \theta(a,b)str(BA)
\end{equation*}
\begin{center}
for all $A \in  Hom(V)_a$ , $B \in Hom(V)_b$ ; $a,b \in \mathcal{K}$.
\end{center}

\begin{definition}
Let L be a Klein graded Lie algebra and let V be a finite-dimensional graded L-module. The n-linear form on L which is defined by 
\begin{equation*}
(A^1,.......,A^n) \to str(A^1_V.......A^n_V)
\end{equation*}
\begin{center}
for all $A^i \in L$ , $1 \leq i \leq n$,
\end{center}
is even and L-invariant (with respect to the adjoint representation of L).
\end{definition}
The n-linear form defined above called the n-linear form associated with the graded L-module V. The most important case is obtained if we choose V=L and n=2. This special case is described below, and it is known as killing form.

\begin{definition}
Let L be a finite-dimensional Lie superalgebra and let $\gamma$ be the automorphism of L such that
\begin{equation*}
\gamma(A) = \theta(a,a) A \text{  for all } A \in L_a ; a \in \mathcal{K}.
\end{equation*}
The bilinear form $\phi$ which is defined by
\begin{equation*}
\phi(A,B) = str(adA adB) = Tr(\gamma adA adB)
\end{equation*}
for all A,B $\in L$. We call $\phi$ the killing form of the Klein graded Lie algebra L. $\phi$ has degree e,it is invariant and "kleinsymmetric".
\end{definition}

\section{Induced representations}\label{inducedsrepresentations}
In this section if not specified, L is a Lie superalgebra or a Klein-graded Lie algebra. When it is necessary we will seperate the cases.

\begin{proposition}
Let L be a graded Lie algebra($Z_2$ or $\mathcal{K}$-grading) and L' a subalgebra of L. Also, let $(E_j)_{j \in J})$ be a family of homogeneous elements of L such that the images of the elements $E_j$ under the canonical mapping $L \to L/L'$ form a basis of the vector space$ L/L'$. We assume that the index set J is totally ordered.

Let H be the set of all finite sequences $(j_1,....,j_r)$ in J such that
\begin{equation*}
\begin{split}
r \geq 0 \text{ arbitrary} \\
j_1 \leq j_2 \leq ... \leq j_r \\
j_p <  j_{p+1} \text{ if } E_{j_p} \text{ and } E_{j_{p+1}} \text{ are odd}.
\end{split}
\end{equation*}
If $N=(j_1,.....,j_r)$ is a sequence of this type we define
\begin{equation*}
E_N = E_{j_1} E_{j_2} ...... E_{j_r}
\end{equation*}
(by convention $E_{\varnothing}$ = 1).

Then the family $(E_N)_{N \in H}$ is a basis of the left as well as of the right $U(L')$-module U(L).
\end{proposition}

Let V be a left U(L')-module. Regarding U(L) as a right graded U(L')-module we can construct the tensor product
\begin{equation*}
\bar{V} = U(L) \otimes_{U(L')} V.
\end{equation*}
$\bar{V}$ has a natural structure of a $Z_2$-graded vector space(resp. Klein graded vector space). The subspace $\bar{V}_x, x \in Z_2$ (resp. x in Klein group) being spanned by the tensors of the form $X\otimes y$ with X $\in  U(L)_a$ y $\in V_b$ ; $a,b \in Z_2$ , a+b = x (resp. a,b $\in \mathcal{K}$, ab=x). Now it is obvious that $\bar{V}$ has a structure of a left graded U(L)-module with satisfies
\begin{equation*}
\begin{split}
X(Y\otimes y) = (XY)\otimes y \\
\text{for all X,Y in U(L), y in V}.
\end{split}
\end{equation*}
This graded U(L)-module $\bar{V}$ is called the graded U(L)-module induced from the graded $U(L')$-module V.

One can define the even(resp. of degree e) linear mapping
\begin{equation*}
a : V \to \bar{V}
\end{equation*}
by
\begin{equation*}
a(y) = 1 \otimes y \text{ , if  } y \in V.
\end{equation*}
The linear mapping a is $U(L')$-invariant:
\begin{equation*}
a(X'y) = X'a(y)     \hspace{10pt} \text{ for all } X' \in U(L') , y \in V.
\end{equation*}

\begin{lemma}
For every sequence $N \in H$ the linear mapping 
\begin{equation*}
V \to \bar{V} \hspace{10pt} , y \to E_N a(y) = E_N \otimes y
\end{equation*}
is injective; in particular, a itself is injective. The vector space $\bar{V}$ is the direct sum of its subspaces $E_N a(V)$ , $N \in H$.
\end{lemma}

This lemma may now be applied to give us a proof of the Ado theorem for both Lie superalgebras and Klein graded Lie algebras. In fact, let $L' = L_0$ (resp. $L' = L_e$) and let V be any $L_0$-module (resp.  $L_e$-module). We introduce a trivial $Z_2$-gradation (resp. Klein-gradation) in V by defining
\begin{equation*}
V_0 = V \hspace{10pt} , \hspace{10pt} V_1 = {0}
\end{equation*}
respectively
\begin{equation*}
V_e = V \hspace{10pt} , \hspace{10pt} V_r = {0}, \hspace{10pt} V_s = {0}, \hspace{10pt} V_t = {0}.
\end{equation*}

Then the induced L-module $\bar{V}$ is well-defined and lemma impLies:

a) If the $L_0$-module(resp.  $L_e$-module) V is faithful, then the L-module $\bar{V}$ is faithfull too, provided that $V \neq \{0\}$.

b) If the vector spaces $ L_1$ (resp. some of $L_r, L_s,L_t$) and V are finite-dimensional, then the same holds true for $\bar{V}$.

Now, using these remarks and the Ado theorem for Lie algebras we conclude that:

\begin{theorem}
Every finite-dimenstional Lie superalgebra(resp. Klein-graded Lie algebra) has a faithful finite-dimensional graded representation.
\end{theorem}

\section{A special case}\label{aspecialcase}
In this section we will examine a graded Lie algebra which has even part equal to a known Lie algebra, more specific we will examine the case of sl(2,C). On the first part we will find some equations which come easily by composing the relations of a base of sl(2,C) and the known relations for an irreducible representation. We will use some tools from the theory of Lie algebras to find results for a grades Lie algebra. On the second part, we will examine the meaning of these equations and we will write some remarks for this algebra.
\subsection{Equations}
Let $L=L_e \bigoplus L_r \bigoplus L_s \bigoplus \L_t$ be a Klein graded Lie Algebra. We set $L_e$ to be the simple Lie algebra sl(2,C). $ L_r$ is a $L_e$-module and as it follows from Weyl theorem,  $L_r$ is direct sum of irreducible  representations of sl(2,C). Lets say $L_r = \bigoplus_{k} M_k, M_k$ is irreducible representation. We select a $M_k$ and we  choose a base for this vector space, lets say  \{$e_0,e_1 ... e_n$\}. The base \{h,x,y\} of sl(2,C) comes with the the following rules:
\begin{equation*}
\begin{split}
[h,x]=2x\\
[h,y]=-2y\\
[x,y]=h
\end{split}
\end{equation*}
The bracket product for two elements of $M_k$ is in $L_e$, so we have:
\begin{equation*}
\begin{split}
[e_i,e_j] = a_{i,j}h + b_{i,j}x + c_{i,j}y \\
\text{for some  }a_{i,j}, b_{i,j}, c_{i,j} \in \mathbb{C}.
\end{split}
\end{equation*}

Now using the Jacobi identity for Klein-graded Lie algebras we have that:
\begin{equation}\label{basicequ1}
\theta(e,r)[h,[e_i,e_j]] + \theta(r,r)[e_j,[h,e_i]] + \theta(r,e)[e_i,[e_j,h]] = 0
\end{equation}
\begin{equation}\label{basicequ2}
\theta(e,r)[x,[e_i,e_j]] + \theta(r,r)[e_j,[x,e_i]] + \theta(r,e)[e_i,[e_j,x]] = 0
\end{equation}
\begin{equation}\label{basicequ3}
\theta(e,r)[y,[e_i,e_j]] + \theta(r,r)[e_j,[y,e_i]] + \theta(r,e)[e_i,[e_j,y]] = 0
\end{equation}

We will work these equations to take some results. Due to equation \eqref{basicequ1} and the rules for an irreducible representation of sl(2,C) we have:

\begin{align*}
0 &= \theta(e,r)[h,[e_i,e_j]] + \theta(r,r)[e_j,[h,e_i]] + \theta(r,e)[e_i,[e_j,h]]
\\ 0 &= \theta(e,r)[h,[e_i,e_j]] - \theta(r,r)\theta(r,r)[[h,e_i],e_j] -\theta(r,e)\theta(r,e)[e_i,[h,e_j]]
\\ 0 &= [h, a_{i,j}h + b_{i,j}x + c_{i,j}y] - (n-2i)[e_i,e_j] -(n-2j)[e_i,e_j] 
\\ 0 &= 2b_{i,j}x - 2c_{i,j}y - (n-2i)(a_{i,j}h + b_{i,j}x + c_{i,j}y) -(n-2j)(a_{i,j}h + b_{i,j}x + c_{i,j}y)
\\ 0 &= 2ha_{i,j}(-n+i+j) +  2xb_{i,j}(1 -n+i +j) +2yc_{i,j}(-1-n+i+j)
\end{align*}

and we obtain the equations 1-3 from the next part.
Now we are going to use the equation \eqref{basicequ2} and we will get equations 4-6:
\begin{align*}
0 &= \theta(e,r)[x,[e_i,e_j]] + \theta(r,r)[e_j,[x,e_i]] + \theta(r,e)[e_i,[e_j,x]]
\\ 0 &= \theta(e,r)[x,[e_i,e_j]] - \theta(r,r)\theta(r,r)[[x,e_i],e_j] -\theta(r,e)\theta(r,e)[e_i,[x,e_j]]
\\ 0 &= [x, a_{i,j}h + b_{i,j}x + c_{i,j}y] - (n-i+1)[e_{i-1},e_j] -(n-j+1)[e_i,e_{j-1}]
\\ 0 &= -2a_{i,j}x + c_{i,j}h - (n-i+1)(a_{i-1,j}h + b_{i-1,j}x  + c_{i-1,j}y)
\\  &-(n-j+1)(a_{i,j-1}h + b_{i,j-1}x + c_{i,j-1}y)
\\ 0 &=h(-a_{i-1,j}(n-i+1) - a_{i,j-1}(n-j+1) + c_{i,j}
\\ &- x(b_{i-1,j}(n-i +1) + b_{i,j-1}(n-j +1) +2a_{i,j})
\\ &- y(c_{i-1,j}(n-i +1) + c_{i,j-1}(n-j +1))
\end{align*}

Finally, we use equation \eqref{basicequ3} and we get equations 7-9:
\begin{align*}
0 &= \theta(e,r)[y,[e_i,e_j]] + \theta(r,r)[e_j,[y,e_i]] + \theta(r,e)[e_i,[e_j,y]]
\\ 0 &= \theta(e,r)[y,[e_i,e_j]] - \theta(r,r)\theta(r,r)[[y,e_i],e_j] -\theta(r,e)\theta(r,e)[e_i,[y,e_j]]
\\ 0 &= [y, a_{i,j}h + b_{i,j}x + c_{i,j}y] - (i+1)[e_{i+1},e_j] -(j+1)[e_i,e_{j+1}]
\\ 0 &= 2a_{i,j}y -b_{i,j}h - (i+1)(a_{i+1,j}h + b_{i+1,j}x + c_{i+1,j}y) 
\\ &-(j+1)(a_{i,j+1}h + b_{i,j+1}x + c_{i,j+1}y)
\\ 0 &=h(-a_{i+1,j}(i+1) - a_{i,j+1}(j+1) - b_{i,j}) 
\\ &- x(b_{i-1,j}(i +1) + b_{i,j-1}(j +1)) 
\\ &+y(-c_{i+1,j}(i +1) - c_{i,j+1}(j +1) +2a_{i,j})
\end{align*}

So after these calculations we have:

\begin{enumerate}
\item $a_{i,j}(n-i-j) = 0$
\item $b_{i,j}(n-i-j-1) = 0 $
\item $c_{i,j}(n-i-j+1) = 0$
\item $a_{i-1,j}(n-i+1) + a_{i,j-1}(n-j+1) - c_{i,j}=0 $
\item $b_{i-1,j}(n-i+1) + b_{i,j-1}(n-j+1) + 2a_{i,j}=0 $
\item $c_{i-1,j}(n-i+1) + c_{i,j-1}(n-j+1) =0 $
\item $a_{i+1,j}(i+1) + a_{i,j+1}(j+1) + b_{i,j} = 0 $
\item $b_{i+1,j}(i+1) + b_{i,j+1}(j+1)                = 0 $
\item $c_{i+1,j}(i+1) + c_{i,j+1}(j+1) - 2a_{i,j} = 0 $
\end{enumerate}

Now we will take a closer look to these relations and make some observations on the structure of $M_k$. From the first 3 relations we will gain some information on the form of the matrices a,b,c and the rest relations will lead us to a relation between matrices a,b and c. In more details, the first relation $a_{i,j}(n-i-j) = 0$ shows that the elements $a_{i,j}$ can be non-zero only in the case that $i+j \neq n$ and they are zero otherwise. The matrix a is anti-diagonal, it has the following form:
\[ a_{i,j}=\left( \begin{array}{ccccc}
0 & \cdots & \cdots & 0 & a_{n,0}\\
\vdots &  & 0& \dots& 0\\
\vdots & 0 &\dots & 0&\vdots\\
0 &\dots & 0& & \vdots\\
a_{0,n} & 0 & \cdots & \cdots& 0
\end{array} \right)\] 

\vspace{10pt}
The second relation $b_{i,j}(n-i-j-1) = 0$ shows that the elements $b_{i,j}$ can be non-zero only in the case that $i+j \neq n-1$ and they are zero otherwise. So the matrix b has the following form:

\[ b_{i,j}=\left( \begin{array}{ccccc}
0 & \cdots& 0 & b_{n-1,0} & 0\\
\vdots &  0& \dots& 0& 0\\
0 & \dots &\dots &  &\vdots\\
b_{0,n-1} &0 & & 0 & \vdots\\
0 & 0 & \cdots & \cdots& 0
\end{array} \right)\] 

\vspace{10pt}
Finally, the third relation $c_{i,j}(n-i-j+1) = 0$ shows that the elements $c_{i,j}$ can be non-zero only in the case that $i+j \neq n+1$ and they are zero otherwise. So the matrix c has the following form:

\[ c_{i,j}=\left( \begin{array}{ccccc}
0 & 0& \cdots & 0 & 0\\
\vdots & & 0& 0&  c_{n,1}\\
\vdots & \dots &\dots & \dots&0\\
0 &0 & \dots& & \vdots\\
0 & c_{1,n} & 0 & \cdots& 0
\end{array} \right)\] 

\vspace{10pt}

Now we know the form of the matrices. We will use the rest relations (4-9) to express $c_{i,j},b_{i,j},a_{i,j}$ as products of $a_{0,n}$. Before we do that, we have to express $b_{i,j}, c_{i,j}$ in terms of a.
Relation 8 by setting i = i-1 and j = j-1 we get:
\begin{equation*}
b_{i-1,j} = \frac{-i}{j}b_{i,j-1}
\end{equation*}
using this on relation 5 and we have:
\begin{align*}
(n-1+1) \frac{-i}{j}b_{i,j-1} + (n-j+1)b_{i,j-1} = -2a_{i,j}
\\ b_{i,j-1}\frac{j(n-j+1) - i(n-i+1)}{j} = -2a_{i,j}
\\b_{i,j-1} = \frac{-2j}{j(n-j+1) - i(n-i+1)}a_{i,j}
\end{align*}
we set now j = j+1:
\begin{equation}\label{bformula}
b_{i,j} = \frac{2(j+1)}{i(n-i+1) -(j+1)(n-j)}a_{i,j+1}
\end{equation}

Relation 6:
\begin{equation*}
c_{i,j+1} = \frac{-(n-j)}{(n-i)}c_{i+1,j}
\end{equation*}
using this on relation 9 and we have:
\begin{align*}
(i+1)c_{i+1,j} - \frac{(j+1)(n-j)}{n-i}c_{i+1,j} = 2a_{i,j}
\\ c_{i+1,j} = \frac{2(n-i)}{(i+1)(n-i) - (j+1)(n-j)}a_{i,j}
\end{align*}
we set i=i-1:

\begin{equation}\label{cformula}
c_{i,j} = \frac{2(n-i+1)}{(i-j-1)(n-i-j)}a_{i-1,j}
\end{equation}

Now we can use the last 2 relations to find a formula for $a_{i,j}$. We use relation 4(or relation 7, it is the same result):

\begin{align*}
0 &=(n-i+1)a_{i-1,j} + (n-j+1)a_{i,j-1} - \frac{2(n-i+1)}{(i-j-1)(n-i-j)}a_{i-1,j} 
\\ a_{i,j-1} &= a_{i-1,j}\frac{(n-i+1)(i-j-1)(n-i-j) -2(n-i+1)}{(j-i+1)(n-i-j)(n-j+1)}
\end{align*}
We set j = n-i+1:

\begin{equation}
a_{i,n-i} = a_{i-1,n-i+1}\frac{(n-i+1)(2i-n)}{i(n-2i+2)}
\end{equation}
The last one for i=1:
\begin{equation*}
a_{1,n-1} = -a_{0,n}(n-2)
\end{equation*}
The last one for i=2:
\begin{align*}
a_{2,n-2} &= -a_{1,n-1}\frac{(n-1)(n-4)}{2(n-2)} 
\\ &= (-1)^2 a_{0,n}\frac{(n-1)(n-4)}{2!}
\end{align*}
The last one for i=3:
\begin{align*}
a_{3,n-3} &= -a_{2,n-2}\frac{(n-2)(n-6)}{3(n-4)} 
\\ &= (-1)^3 a_{0,n}\frac{(n-1)(n-2)(n-6)}{3!}
\end{align*}
The last one for i=4:
\begin{align*}
a_{4,n-4} &= -a_{3,n-3}\frac{(n-3)(n-8)}{4(n-6)} 
\\&= (-1)^4 a_{0,n}\frac{(n-1)(n-2)(n-3)(n-8)}{4!}
\end{align*}
if we continue these calculations, we can notice that the formula is:
\begin{equation}\label{aequation}
a_{i,n-i} =(-1)^i \binom{n}{i} \frac{n-2i}{n} a_{0,n}
\end{equation}
Using the last one and by setting  j = n-i-1 in \eqref{bformula}, we get:
\begin{equation*}
b_{i,n-i-1} = \frac{2(n-i)}{2i-n}a_{i,n-i}
\end{equation*}
or
\begin{equation}
b_{i,n-i-1} =(-1)^{i+1} \binom{n}{i} \frac{2(n-i)}{n} a_{0,n}
\end{equation}
Now using \eqref{aequation} and by setting  j = n-i+1 in \eqref{cformula}, we get:
\begin{equation*}
c_{i,n-i-1} = \frac{2(n-i+1)}{-(2i-n-2)}a_{i-1,n-i+1}
\end{equation*}
or
\begin{equation}
c_{i,n-i} =(-1)^{i-1} \binom{n}{i-1} \frac{2(n-i+1)}{n} a_{0,n}
\end{equation}

\subsection{The Meaning}
The calculations are done. Now we may examine what results we can extract out of these equations. Consider these two products
\begin{equation*}
[e_i,e_j] = a_{i,j}h + b_{i,j}x + c_{i,j}y
\end{equation*}
and 
\begin{equation*}
[e_j,e_i] = a_{j,i}h + b_{j,i}x + c_{j,i}y.
\end{equation*}
Using the bracket properties, we have that $a_{i,j} = -\theta(r,r)a_{j,i}$. 
So, if  $\theta(r,r)=1$ then $L_e \bigoplus L_r$ is a Lie algebra and we have that  $a_{i,j} = -a_{j,i}$. Equation \eqref{aequation} says that $a_{n,0} = (-1)^{n+1}a_{0,n}$ and now we see that if $L_e \bigoplus L_r$ is a Lie algebra, then n has to be even. Consequently $dimM_k$ has to be odd. On the other hand, if $\theta(r,r)=-1$ then $L_e \bigoplus L_r$ is a Lie superalgebra. In that case n has to be odd so $dimM_k$ is even.

\begin{remark}
If $\theta(r,r)=1$ then $dimM_k$ is odd, otherwise $dimM_k$ is even.
\end{remark}

\begin{remark}
We can see that if one of $b_{i,j}$, $c_{i,j}$ or $a_{i,j}$  is zero, then all of them are zero, so the matrices a,b,c are zero. Hence $[M_k,M_k] =\{0\}$.
\end{remark}

\begin{remark}

If the dimension of $M_k$ is odd, then n is even. In that case, in equation \eqref{aequation} there is an integet i such that n-2i=0, so $a_{i,n-i}  = 0$ and due to previous remark we have that $[M_k,M_k] =\{0\}$.
\end{remark}

\begin{remark}
Let  $L_e \bigoplus L_r$ be a Lie algebra, we know that $a_{i,n-i} =(-1)^i \binom{n}{i} \frac{n-2i}{n} a_{0,n}$ and recall that on Lie algebra case n is even. Due to remark, $[M_k,M_k] =\{0\}$.
\end{remark}

\begin{proposition}
Consider 2 different irreducible representation of $L_e$, $M_k, M_l \in L_r$ with $ k \neq l $, then $[M_k,M_l] = 0$.
\end{proposition}

\begin{proof}

First we show that  $[M_k,M_l]$ is an ideal of $L_e$. Let $a \in L_e, b \in M_k, c \in M_l$ then we have:
\begin{equation*}
[a,[b,c]] = [[a,b],c] + [b,[a,c]]
\end{equation*}
we see that 
\begin{equation*}
[[a,b],c] \in [M_k,M_l] \text{ and } [b,[a,c]]  \in [M_k,M_l].
\end{equation*}
Hence,
\begin{equation*}
[L_e, [M_k, M_l]] \subset [M_k,M_l]
\end{equation*}
and $[M_k,M_l]$ is an ideal of $L_e$.
Now, assume that $[M_k,M_l]  =  L_e$ ; we have 2 cases($[M_k,M_k]$ is also an ideal),

Case 1:  If one assume that $[M_k,M_k] = \{0\}$ then
\begin{equation*}
[[M_k,M_l], M_k] = M_k
\end{equation*}
on the other hand, using jacobi identity, 
\begin{equation*}
[[M_k,M_l], M_k] = [[M_k,M_k], M_l] = \{0\} 
\end{equation*}
which is a contradiction. 

Case 2:  If one assume that $[M_k,M_k] = L_e$ then
\begin{equation*}
[[M_k,M_l], M_k] = M_k
\end{equation*}
on the other hand, using jacobi identity, 
\begin{equation*}
[[M_k,M_l], M_k] = [[M_k,M_k], M_l] = M_l
\end{equation*}
which is a contradiction. 

On both cases it is a condradiction, so our first assumption that $[M_k,M_l]  =  L_e$ was wrong and we have that $[M_k,M_l]  =  \{0\}$. 
\end{proof}

\begin{proposition}
Consider an  irreducible representation of $L_e$, $M_k \in L_r$ and assume that $[M_k,M_k] = L_e$ then $L_r = M_k$.
\end{proposition}

\begin{proof}
If there is $M_l \in L_r$, $M_l \neq M_k$ and $M_l \neq \{0\} $ then:
\begin{equation*}
[[M_k,M_l], M_k] = \{0\}
\end{equation*}
 on the other hand, using jacobi identity, 
\begin{equation*}
[[M_k,M_l], M_k] = [[M_k,M_k], M_l] = M_l
\end{equation*}
which is a contradiction. 
\end{proof}

Using these remarks and propositions we are ready to classify some cases.
\vspace{10pt}
\newline
\textbf{Case 1}:
L = $L_e \bigoplus L_r$ is a Lie superalgebra. 

\textbf{Subcase 1}: L = $L_e \bigoplus M$, [M,M] = $L_e$

\textbf{Subcase 2}: L = $L_e \bigoplus \sum_k M_k$, $[M_k,M_l] = \{0\}$
\newline
On subcase 2 we have to notice that $W = \sum_k M_k$ is a Grassman algebra.
\vspace{10pt}
\newline
\textbf{Case 2}:
L = $L_e \bigoplus L_r$ is a Lie algebra. We can write L = $L_e \bigoplus \sum_k M_k$ and based on the remark, on this case we have that $[M_k,M_l] = \{0\} $.

\begin{remark}
On case 2, the ideal W = $ \sum_k M_k$ is solvable(because it's abelian) moreover, W is the the maximal solvable ideal of L i.e. W is the radical of the Lie algebra L. We can also write $L_e \sim L/Rad(L)$.
\end{remark}
\begin{proof}
Let denote with R the radical of L. Assume that $R = N \oplus W$, then N is a solvable ideal and it is a subset of $L_e$. Consequently,
\begin{equation*}
 [L,N] \subset N
\end{equation*}
and
\begin{equation*}
 [W,N] = 0 \text{ and } [Le,N] \subset N
\end{equation*}
so N is $L_e$-module. Hence, N = \{0\} or N = $L_e$. If N=\{0\} then the remark is proven, otherwise N = $L_e$ and L is a solvable. But if L is solvable, then the even part ( $L_e$ in our case ) is solvable, with is a contradiction.
\end{proof}

\section{Example: The Relative Parabose Set}\label{RPSexample}

As an example we examine The Relative Parabose Set as described in \cite{axioms}, \cite{ref5}. $P_{BF}$ is generated, as an associative algebra,  by the generators $B_i^{\xi}, F_j^{\eta}$  , for all values $i, j = 1,2,\dots$ and $\xi,\eta = \pm1$. The relations satisfied by the above generators are:

The usual trilinear relations of the parabosonic and the parafermionic algebras which can be
compactly summarized as
\begin{equation}\label{rps1}
 [\{B_i^{\xi},B_j^{\eta}\},B_k^{\epsilon}] = (\epsilon - \eta)\delta_{jk}B_i^{\xi} +(\epsilon - \xi)\delta_{ik}B_j^{\eta}
\end{equation}
\begin{equation}\label{rps2}
 [\{F_i^{\xi},F_j^{\eta}\},F_k^{\epsilon}] = \frac{1}{2}(\epsilon - \eta)^2 \delta_{jk}F_i^{\xi} -\frac{1}{2}(\epsilon - \xi)^2 \delta_{ik}F_j^{\eta}
\end{equation}
for all values $i, j,k = 1,2,\dots$ and $\xi,\eta, \epsilon = \pm1$ together with the mixed trilinear relations
\begin{equation}\label{rps3}
 [\{B_k^{\xi},B_l^{\eta}\},F_m^{\epsilon}] = [\{F_k^{\xi},F_l^{\eta}\},B_m^{\epsilon}]  =0
\end{equation}
\begin{equation}\label{rps4}
[\{F_k^{\xi},B_l^{\eta}\},B_m^{\epsilon}] = (\epsilon - \eta)\delta_{lm}F_k^{\xi} \; \; , \; \; \; 
 [\{B_k^{\xi},F_l^{\eta}\},F_m^{\epsilon}] = \frac{1}{2}(\epsilon - \eta)^2 \delta_{lm}B_k^{\xi}
\end{equation}
for all values $i, j,k = 1,2,\dots$ and $\xi,\eta, \epsilon = \pm1$ , which represent a kind of algebraically established interaction between parabosonic and parafermionic elements and characterize the relative parabose set.

We use $B^{-1},B^{1},F^{-1},F^{1}$ as symbols, there is no connection with the multiplicative inverse. One can easily observe that the relations \ref{rps1}, \ref{rps2} involve only the parabosonic and the parafermionic degrees of freedom separately while the interaction relations \ref{rps3}, \ref{rps4} mix the parabosonic with the parafermionic degrees of freedom.

In all the above and in what follows, we use the notation $[x,y]$ (i.e.: the "commutator") to imply the expression $xy-yx$ and the notation $\{x,y\}$ (i.e.: the "anticommutator") to imply the expression $xy+yx$, for x and y any elements of the algebra $P_{BF}$.

Let $L=L_e \oplus L_s \oplus L_t \oplus L_r$  and define 
\begin{align*}
L_s &= span(B_i^{1},B_j^{-1} | i,j = 1,2, \dots)
\\ L_t &= span(F_i^{1},F_j^{-1} | i,j = 1,2, \dots)
\\ L_r &= span(\{B_i^{m},F_j^{n}\} \; \; | \; \;m,n = \pm 1 \; \;,\; \; i,j = 1,2, \dots)
\\ L_e &= span(\{B_i^{m},B_j^{n}\},[F_a^{m},F_b^{n}] \; \;| \; \; m,n = \pm1 \; \; , \; \; a,b,i,j = 1,2, \dots)
\end{align*}
Hence, if we keep in mind \ref{rps1} - \ref{rps4}, L is a Klein graded Lie algebra. Finally, if we recall that a Klein graded Lie algebra's color is one of the $\theta_1,\theta_2,\theta_3,\theta_4$ we discussed in section \ref{kleingraded}. In this construction we can see that the color function is the $\theta_3$ of section \ref{kleingraded}.

\end{document}